\documentclass[12pt]{amsart}  

\usepackage{amssymb}
\usepackage{amscd}

\newtheorem{thm}{Theorem}[subsection]

\newtheorem{prop}[thm]{Proposition}
\newtheorem{lemma}[thm]{Lemma}

\newtheorem{definition}[thm]{Definition}
\newtheorem{proposition}[thm]{Proposition}

\theoremstyle{remark}
\newtheorem{remark}[thm]{Remark}
\newtheorem{example}[thm]{Example}

\theoremstyle{definition}

\numberwithin{equation}{section}

\newcommand{\cal}{\mathcal}

\newcommand{\VM}{{\mathcal{V}}^{\bullet}(M)}
\newcommand{\LD}{{\rm{LD}}}
\newcommand{\LC}{{\rm{LC}}}
\newcommand{\LfC}{{\rm{LfC}}}
\newcommand{\C}{{\Bbb C}}
\newcommand{\bbA}{{\Bbb A}}

\newcommand{\bbR}{{\Bbb R}}
\newcommand{\End}{\operatorname{End}}
\newcommand{\Tor}{\operatorname{Tor}}

\newcommand{\ad}{\operatorname{ad}}
\newcommand{\FMC}{\operatorname{FMC}}
\newcommand{\FMfC}{\operatorname{FMfC}}
\newcommand{\FreeOp}{\operatorname{FreeOp}}
\newcommand{\Coker}{\operatorname{Coker}}
\newcommand{\PaB}{\mathbf{PaB}}
\newcommand{\PaP}{\mathbf{PaP}}

\newcommand{\frt}{\mathfrak{t}}

\newcommand{\As}{\operatorname{As}}
\newcommand{\Gerst}{\operatorname{Gerst}}
\newcommand{\Nerve}{\operatorname{Nerve}}
\newcommand{\BV}{\operatorname{BV}}
\newcommand{\Ob}{\operatorname{Ob}}
\newcommand{\Com}{\operatorname{Com}}
\newcommand{\Lie}{\operatorname{Lie}}

\newcommand{\cPsh}{{\cP^{\#}}}

\newcommand{\bbZ}{{\Bbb Z}}

\newcommand{\cO}{{\mathcal O}}
\newcommand{\cP}{{\mathcal P}}

\newcommand{\cV}{{\mathcal V}}
\newcommand{\cA}{{\mathcal A}}
\newcommand{\cB}{{\mathcal B}}

\newcommand{\cQ}{{\mathcal Q}}
\newcommand{\cC}{{\mathcal C}}

\newcommand{\cR}{{\mathcal R}}

\newcommand{\Li}{{L_{\infty}}}
\newcommand{\Ai}{{A_{\infty}}}

\newcommand{\iHom}{\underline{\operatorname{Hom}}}
\newcommand{\Cai}{{\operatorname{Calc}_{\infty}}}

\newcommand{\Ca}{{\operatorname{Calc}}}

\newcommand{\eq}{\sim}

\newcommand{\Br}{{\operatorname{Bar}}}
\newcommand{\Cbr}{{\operatorname{Cobar}}}

\newcommand{\Ker}{\operatorname{Ker}}

\newcommand{\op}{\operatorname{op}}

\newcommand{\D}{\operatorname{D}}
\newcommand{\R}{{\Bbb R}}
\newcommand{\Q}{{\Bbb Q}}
\newcommand{\Hom}{\operatorname{Hom}}

\newcommand{\isomoto}{\overset{\sim}{\to}}
\newcommand{\lisomoto}{\overset{\sim}{\leftarrow}}
\newcommand{\id}{\operatorname{id}}
\newcommand{\Tr}{\operatorname{Tr}}
\newcommand{\FM}{\operatorname{FM}}
\newcommand{\Conf}{\operatorname{Conf}}

\newcommand{\g}{{\mathfrak{g}}}

\newcommand{\co}{\bf{{c}}}
\newcommand{\cha}{\bf{{h}}}

\newcommand{\Y}{Y}
\newcommand{\Ccc}{C^{\bullet}(A)}

\newcommand{\Hcc}{H^{\bullet}(A)}

\makeatletter
\renewcommand{\subsubsection}{\@startsection
{subsubsection}%
{2}%
{0mm}%
{-\baselineskip}%
{-0.5\baselineskip}%
{\normalfont\normalsize\bfseries }}%
\makeatother

\begin{document}

\copyrightinfo{~2012}{Boris Tsygan}
\setcounter{page}{19}

\title{Noncommutative Calculus and Operads}

\author[B.Tsygan]{Boris Tsygan}
\address{Department of
Mathematics, Northwestern University, Evanston, IL 60208-2730, USA} \email{b-tsygan@northwestern.edu}

\thanks{
The author was partially supported by NSF grant
DMS-0906391}




\maketitle

\section{Introduction}
This expository paper is based on lecture courses that the author taught at the Hebrew University of Jerusalem in the year of 2009--2010 and at the Winter School on Noncommutative Geometry at Buenos Aires in July-August of 2010. It gives an overview of works on the topics of noncommutative calculus, operads and index theorems. 

Noncommutative calculus is a theory that defines classical algebraic structures arising from the usual calculus on manifolds in terms of the algebra of functions on this manifold, in a way that is valid for any associative algebra, commutative or not. It turns out that noncommutative analogs of the basic spaces arising in calculus are well-known complexes from homological algebra. For example, the role of noncommutative multivector fields is played by the Hochschild cochain complex of the algebra; the the role of noncommutative forms is played by the Hochschild chain complex, and the role of the noncommutative de Rham complex by the periodic cyclic complex of the algebra. These complexes turn out to carry a very rich algebraic structure, similar to the one carried by their classical counterparts. Moreover, when the algebra in question is the algebra of functions, the general structures from noncommutative geometry are equivalent to the classical ones. These statements rely on the Kontsevich formality theorem \cite{Konts} and its analogs and generalizations. We rely on the method of proof developed by Tamarkin in \cite{T}, \cite{T1}. The main tool in this method is the theory of operads \cite{May}.

A consequence of the Kontsevich formality theorem is the classification of all deformation quantizations \cite{BFFLS} of a given manifold. Another consequence is the algebraic index theorem for deformation quantizations. This is a statement about a trace of a compactly supported difference of projections in the algebra of matrices over a deformed algebra. It turns out that all the data entering into this problem (namely, a deformed algebra, a trace on it, and projections in it) can be classified using formal Poisson structures on the manifold. The answer is an expression very similar to the right hand side of the Atiyah-Singer index theorem. For a deformation of a symplectic structure, all the results mentioned above were obtained by Fedosov \cite{Fe}; they imply the Atiyah-Singer index theorem and its various generalizations \cite{BNT}. 

The algebraic index theorem admits a generalization for deformation quantizations of complex analytic manifolds. In this new setting, a deformation quantization as an algebra is replaced by a deformation quantization as an algebroid stack, a trace by a Hochschild cocycle, and a difference of two projections by a perfect complex of (twisted) modules. The situation becomes much more mysterious than before, because both the classification of the data entering into the problem and the final answer depend on a Drinfeld associator \cite{Dr}. The algebraic index theorem for deformation quantization of complex manifolds in its final form is due to Willwacher (\cite{Wil}, \cite{Wil1}, and to appear).

The author is greatly indebted to the organizers and the members of the audience of his course and talks in Jerusalem for the wonderful stimulating atmosphere. He is especially thankful to  Volodya Hinich and David Kazhdan with whom he had multiple discussions on the subject and its relations to the theory of infinity-categories, as well as to Ilan Barnea for a key argument on which Section \ref{s:Infinity-algebras and categories} is based. He is grateful to Oren Ben Bassat, Emmanouil Farjoun, Jake Solomon, Ran Tesler and Amitai Zernik for very interesting and enjoyable discussions. It is my great pleasure to thank Willie Corti\~nas and of the other co-organizers of the Winter School in Buenos Aires, as well as the audience of my lectures there.

\section{ Hochschild and cyclic homology of algebras}
\label{section:CC}
\label{subsection:C-CC}
Let $k$ denote a commutative algebra over a field of characteristic zero
and let $A$ be a flat $k$-algebra with unit,
not necessarily commutative. Let $\overline A = A/k\cdot 1$. For $p\geq 0,$ let
$C_p(A)\overset{def}{=}A\otimes_k\overline A^{\otimes_k p}$. Define
\begin{eqnarray}\label{diffl:b}
b : C_p(A) & \rightarrow & C_{p-1}(A) \\ \nonumber
a_0\otimes\ldots\otimes a_p & \mapsto &
(-1)^p a_pa_0\otimes\ldots\otimes a_{p-1} + \\ & &
\sum_{i=0}^{p-1}(-1)^ia_0\otimes\ldots\otimes a_ia_{i+1}\otimes\ldots\otimes
a_p\ .\nonumber
\end{eqnarray}
Then $b^2 = 0$ and one gets the complex $(C_\bullet, b)$, called {\em the
standard
Hochschild complex of $A$}. The homology of this complex is denoted by
$H_{\bullet}(A,A)$, or by $HH_{\bullet}(A)$.

\begin{proposition} \label{prop:b,B}
The map
\begin{eqnarray}\label{diffl:B}
B : C_p(A) & \rightarrow & C_{p+1}(A) \\ \nonumber
a_0\otimes\ldots\otimes a_p & \mapsto & \sum_{i=0}^p (-1)^{pi}
1\otimes a_i\otimes\ldots\otimes a_p\otimes a_0\otimes\ldots\otimes a_{i-1}
\end{eqnarray}
satisfies $B^2 = 0$ and $bB+Bb = 0$ and therefore defines a map of complexes
\[
B: C_\bullet(A) \rightarrow C_\bullet(A)[-1]\
\]
\end{proposition}
\begin{definition} \label{dfn:cc}
For $p\in\bbZ$ let
\begin{eqnarray*}
CC^-_p(A) & = & \prod_{\overset{i\geq p}{i \equiv  p \mod 2}} C_{i}(A) \\
CC^{{\rm{per}}}_p(A) & = & \prod_{i \equiv p \mod 2} C_{i}(A)\
\end{eqnarray*}
$$CC_p(A)  =  \bigoplus_{\overset{i\leq p}{i \equiv p \mod 2}} C_{i}(A)$$
\end{definition}
Since $i\geq 0,$ the third formula has a finite sum in the right hand side.
The complex $(CC^-_\bullet(A),B+b)$ (respectively
$(CC^{{\rm{per}}}_\bullet(A),B+b)$, respectively $(CC_\bullet(A),B+b)$) is called
the {\em negative cyclic} (respectively
{\em periodic cyclic}, respectively {\em cyclic}) complex of $A$. The
homology of these complexes is denoted by $HC^-_\bullet(A)$ (respectively by
$HC^{{\rm{per}}}_\bullet(A)$, respectively by $HC_\bullet(A))$.

There are inclusions of complexes
\begin{equation}\label{inclusions}
CC^-_\bullet(A)[-2]\hookrightarrow CC^-_\bullet(A)\hookrightarrow
CC^{per}_\bullet(A)
\end{equation}
and the short exact sequences
\begin{equation}\label{ses:CC-C}
0 \rightarrow CC^-_\bullet(A)[-2] \rightarrow CC^-_\bullet(A) \rightarrow C_\bullet(A) \rightarrow 0\
\end{equation}
\begin{equation}\label{ses:CC-C1}
0 \rightarrow  C_\bullet(A) \rightarrow CC_\bullet(A) \stackrel{S}{\rightarrow} CC_\bullet(A) [2] \rightarrow 0\
\end{equation}
To the double complex $CC_{\bullet}(A)$ one associates the
spectral sequence
\begin{equation} \label{eq:ssccc}
E^2_{pq} = H_{p-q}(A, A)
\end{equation}
converging to $HC_{p+q}(A)$.

In what follows we will use the notation of Getzler and Jones (\cite{GJ}).
Let $u$ denote a variable of degree $-2$. 
\begin{definition}\label{dfn:((u))}
For any $k$-module $M$ we denote by $M[u]$ $M$-valued polynomials in $u$, by $M[[u]]$ $M$-valued power series, and by $M((u))$ $M$-valued Laurent series in $u$.
\end{definition}
The negative and periodic
cyclic complexes are described by the following formulas:
\begin{eqnarray} \label{ses:gejo}
CC^-_{\bullet}(A) & = & (C_{\bullet}(A)[[u]], b + uB) \\
\label{ses:gejoper}
CC^{\text{per}}_{\bullet}(A) & = & (C_{\bullet}(A)((u)), b + uB)\\
CC_{\bullet}(A) & = & (C_{\bullet}(A)((u)) / uC_{\bullet}(A)[[u]], b
+ uB)
\end{eqnarray}
In this language, the map $S$ is just multiplication by $u$.
\begin{remark} \label{rmk:units}
For an algebra $A$ without unit, let $\tilde{A} = A + k\cdot 1$ and put
$$CC_{\bullet}(A) = \operatorname{Ker}(CC_{\bullet}(\tilde{A}) \rightarrow
CC_{\bullet}(k));$$
similarly for the negative and periodic cyclic complexes. If $A$
is a unital algebra then these complexes are quasi-isomorphic to
the ones defined above.
\end{remark}

\subsection{Homology of differential graded algebras}  \label{ss:ccdga}
One can easily generalize all the above constructions to the case when $A$
is
a differential graded algebra (DGA) with the differential $\delta$ (i.e.
$A$ is a graded algebra and $\delta$ is a derivation of degree $1$ such that
$\delta^2=0$).

The action of $\delta$ extends to an action on Hochschild chains by the
Leibniz rule:
\[
\delta (a_0\otimes\ldots\otimes a_p ) =
\sum_{i=1}^p {(-1)^{\sum_{k<i}{(| a_k| + 1)+1}}
(a_0\otimes\ldots\otimes\delta a_i \otimes ldots \otimes a_p )}
\]
The maps $b$ and $B$ are modified to include signs:
\begin{equation} \label{eq:b grad}
b(a_0 \otimes ldots \otimes a_p) = \sum _{k=0}^{p-1} (-1)^{\sum_{i=0}^{k}
{(|a_i| + 1)+1}}
a_0 \otimes ldots \otimes a_k a_{k+1} \otimes ldots \otimes a_p
\end{equation}
$$+ (-1)^{|a_p| + (|a_p|+1)\sum_{i=0}^{p-1}(|a_i|+1)} a_pa_0 \otimes ldots
\otimes a_{p-1}
$$
\begin{equation} \label{eq: B graded}
B(a_0 \otimes ldots \otimes a_p) = \sum_{k=0}^p (-1)^{\sum _{i \leq
k}(|a_i| + 1) \sum _{i \geq k}(|a_i| + 1)} 1 \otimes a_{k+1} \otimes ldots
\otimes a_p \otimes
\end{equation}
$$\otimes a_0 \otimes ldots \otimes a_k $$

The complex $C_{\bullet}(A)$ now becomes the total complex of
the double complex with the differential $b + \delta$:
$$C_p(A)=\bigoplus _{j-i=p}(A\otimes {\overline{A}}^{\otimes j})^{i}$$
The negative and
the periodic cyclic complexes are defined as before in terms of the new
definition of $C_{\bullet}(A)$.
All the results of this section extend to the differential graded case.
\begin{remark}\label{rmk:sums vs products}
Note that the total complex consists of direct sums rather than direct products. This choice, as well as the choice of defining the periodic cyclic complex using Laurent series, is made so that a quasi-isomorphism of DG algebras would induce a quasi-isomorphism of corresponding complexes.
\end{remark}

%
%
%

\subsection{The Hochschild cochain complex} \label{hocochain}

Let $A$ be a graded algebra with unit over a commutative unital
ring $k$ of characteristic zero.  A Hochschild $d$-cochain is a
linear map $A^{\otimes d}\to A$.  Put, for $d\geq 0$,
\begin{equation}\label{eq:Hochcochains1}
        C^d(A) = C^d (A,A) =\operatorname{Hom}_k({\overline{A}}^{\otimes
d},A)
\end{equation}
where ${\overline{A}}=A/k\cdot 1$. Put
\begin{equation}\label{eq:Hochcochains2}
        |D|\;=({\rm {degree\; of\; the\; linear\; map\; }}D)+d
\end{equation}
  Put for cochains $D$ and $E$ from $C^{\bullet}(A,A)$
\begin{equation}\label{eq:Hochcochains3}
        (D\smile E)(a_1,\dots,a_{d+e})=(-1)^{| E|\sum_{i \leq d}(|a_i| + 1)}
        D(a_1,\dots,a_d)\times
\end{equation}
\begin{equation}\label{eq:Hochcochains4}
        \times E(a_{d+1},\dots,a_{d+e});
\end{equation}
\begin{equation}\label{eq:Hochcochains5}
        (D\circ E)(a_1,\dots,a_{d+e-1})=\sum_{j \geq 0}
        (-1)^{(|E|+1)\sum_{i=1}^{j}(|a_i|+1)}\times
        \end{equation}
$$
\times D(a_1,\dots,a_j,
        E(a_{j+1},\dots,a_{j+e}),\dots);
$$
\begin{equation}\label{eq:Hochcochains7}
        [D, \; E]= D\circ E - (-1)^{(|D|+1)(|E|+1)}E\circ D
\end{equation}
These operations define the graded associative algebra
$(C^{\bullet}(A,A)\;,\smile)$ and the graded Lie algebra
($C^{\bullet + 1}(A,A)$, $[\;,\;]$) (cf. \cite{CE}; \cite{G}).
Let
\begin{equation}\label{eq:Hochcochains8}
        m(a_1,a_2)=(-1)^{| a_1|}\;a_1 a_2;
\end{equation}
this is a 2-cochain of $A$ (not in $C^2$).  Put
\begin{equation}\label{eq:Hochcochains9}
        \delta D=[m,D];
\end{equation}
\begin{equation}\label{eq:Hochcochains10}
        (\delta D)(a_1,\dots,a_{d+1})=(-1)^{|a_1||D|+|D|+1}\times
\end{equation}
\begin{equation}\label{eq:Hochcochains11}
        \times a_1 D(a_2,\dots,a_{d+1})+
\end{equation}
$$
        +\sum
_{j=1}^{d}(-1)^{|D|+1+\sum_{i=1}^{j}(|a_i|+1)}
                D(a_1,\dots,a_ja_{j+1},\dots,a_{d+1})
$$
$$
        +(-1)^{|D|\sum_{i=1}^{d }(|a_i|+1)}D(a_1,\dots,a_d)a_{d+1}
$$

One has
\begin{equation}\label{eq:Hochcochains12}
        \delta^2=0;\quad\delta(D\smile E)=\delta D\smile E+(-1)^{|D|}
                D\smile\delta E
\end{equation}
\begin{equation}\label{eq:Hochcochains13}
        \delta[D,E]=[\delta D,E]+(-1)^{|D|+1}\;[D,\delta E]
\end{equation}
($\delta^2=0$ follows from $[m,m]=0$).

Thus $ C^{\bullet}(A,A)$ becomes a complex; we will denote it also by
$C^{\bullet}(A)$. The cohomology of this complex
is $H^{\bullet}(A,A)$ or the Hochschild cohomology. We denote it also by
$H^{\bullet}(A) $.  The $\smile$ product induces the
Yoneda product on $H^{\bullet}(A,A)={\rm{Ext}}_{A\otimes A^0}^{\bullet}(A,A)$.  The
operation
$[\;,\;]$ is the Gerstenhaber bracket \cite{G}.

If $(A, \; \partial)$ is a differential graded algebra then one can define
the differential $\partial$ acting on $\Ccc$ by:
\begin{equation}\label{eq:Hochcochains14}
\partial D \;\; = \; [\partial , D]
\end{equation}

\begin{thm} \cite{G} The cup product and the Gerstenhaber bracket induce a
Gerstenhaber algebra structure on
$\Hcc$ (cf. \ref{ss:gerstenhaber} for the definition of a
Gerstenhaber algebra).
\end{thm}
For  cochains $D$ and $D_i$ define a new Hochschild cochain by the following
formula of Gerstenhaber (\cite{G}) and Getzler (\cite{Ge1}):

\begin{equation}\label{eq:Hochcochains15}
D_0\{D_1, \ldots , D_m\}(a_1, \ldots, a_n) =
\end{equation}
$$
=\sum (-1)^{ \sum_{k\leq {i_p}}(|a_k| + 1)(| D_p|+1)}  D_0(a_1, \ldots
,a_{i_1} , D_1 (a_{i_1 + 1}, \ldots ),\ldots ,
$$
$$D_m (a_{i_m + 1}, \ldots ) , \ldots)
$$
\begin{proposition}
One has
$$
(D\{E_1, \ldots , E_k \})\{F_1, \ldots, F_l \}=\sum (-1)^{\sum _{q \leq
i_p}(|E_p|+1)(|F_q|+1)} \times
$$
$$
\times D\{F_1, \ldots , E_1 \{F_{i_1 +1}, \ldots , \} , \ldots ,  E_k
\{F_{i_k +1}, \ldots , \}, \ldots, \}
$$
\end{proposition}

The above proposition can be restated as follows.
For a cochain $D$ let $D^{(k)}$ be the following $k$-cochain of $\Ccc$:
$$
D^{(k)}(D_1, \ldots, D_k) = D\{D_1, \ldots, D_k\}
$$
\begin{proposition} \label{etoee}
The map
$$
D \mapsto \sum_{k \geq 0} D^{(k)}
$$
is a morphism of differential graded algebras
$$ C^{\bullet}(A) \rightarrow C^{\bullet}(  C^{\bullet}(A))$$
\end{proposition}
\subsection{Products on Hochschild and cyclic complexes} \label{ss:products}
Unless otherwise specified, the reference for this subsection is
\cite{L}.
\subsubsection{Product and coproduct; the K\"{u}nneth exact sequence}
\label{sss:prodcoprod}
For an algebra $A$ define the shuffle product
\begin{equation} \label{eq:sh for one}
\operatorname{sh}: C_p(A) \otimes C_q(A) \rightarrow C_{p+q}(A)
\end{equation}
as follows.
\begin{equation} \label{eq:sh for one, formula}
(a_0 \otimes \ldots \otimes a_p) \otimes (c_{0} \otimes ldots \otimes
c_{q})=
a_0c_0 \otimes \operatorname{sh}_{pq}(a_1, ldots, a_p, \,c_1, ldots, c_q)
\end{equation}
where
\begin{equation} \label{eq:Sh}
\operatorname{sh}_{pq}(x_1, ldots, x_{p+q}) = \sum_{\sigma \in
\operatorname{Sh}(p,q)}\operatorname{sgn}(\sigma)x_{\sigma ^{-1}1}\otimes
ldots \otimes x_{\sigma ^{-1}(p+q)}
\end{equation}
and
$$\operatorname{Sh}(p,q) = \{\sigma \in \Sigma _{p+q} \,| \sigma 1 < \ldots
< \sigma p;\;\sigma (p+1) < ldots  < \sigma (p+q) \} $$

In the graded case, $\operatorname{sgn}(\sigma)$ gets replaced by
the sign computed by the following rule: in all transpositions,
the parity of $a_i$ is equal to $|a_i|+1$ if $i > 0$, and
similarly for $c_i$. A transposition contributes a product of
parities.

The shuffle product is not a morphism of complexes unless $A$ is
commutative. It defines, however, an exterior product as shown in the
following theorem. For two unital algebras $A$ and $C$, let $i_A$, $i_C$ be
the embeddings $a \mapsto a\otimes 1$, resp. $c \mapsto 1\otimes c$ of $A$,
resp. $C$, to $A \otimes C$. We will use the same notation for the
embeddings that $i_A$, $i_C$ induce on all the chain complexes considered by
us.

\begin{thm} \label{thm:exterior product}
For two unital algebras $A$ and $C$ the composition
$$ C_p(A) \otimes C_q(C) \stackrel{i_A \otimes i_C}{\longrightarrow}
C_p(A\otimes C) \otimes C_q(A \otimes C)
\stackrel{\operatorname{sh}}{\longrightarrow} C_{p+q}(A \otimes C)
$$
defines a quasi-isomorphism
$$ {\bf {\operatorname{sh}}}: C_{\bullet}(A) \otimes C_{\bullet}(C) \rightarrow
C_{\bullet}(A \otimes C) $$
\end{thm}
To extend this theorem to cyclic complexes, define
\begin{equation} \label{eq:sh' for one}
\operatorname{sh}': C_p(A) \otimes C_q(A) \rightarrow C_{p+q+2}(A)
\end{equation}
as follows.
\begin{equation} \label{eq:sh` for one, formula}
(a_0 \otimes \ldots \otimes a_p) \otimes (c_{0} \otimes \ldots \otimes
c_{q})\mapsto
1 \otimes \operatorname{sh}'_{p+1,\,q+1}(a_0, \ldots, a_p, \,c_0, \ldots,
c_q)
\end{equation}
where
\begin{equation} \label{eq:Sh'}
\operatorname{sh}'_{p+1,q+1}(x_0, \ldots, x_{p+q+1}) = \sum_{\sigma \in
\operatorname{Sh}'(p+1,q+1)}\operatorname{sgn}(\sigma)x_{\sigma
^{-1}0}\otimes \ldots \otimes x_{\sigma ^{-1}(p+q+1)}
\end{equation}
and $\operatorname{Sh}'(p+1,q+1)$ is the set of all permutations $
\sigma \in \Sigma _{p+q+2}$ such that $\sigma 0 < \ldots  < \sigma
p$, $\sigma (p+1) < \ldots  < \sigma (p+q+1)$, and $ \sigma 0 <
\sigma (p+1) $.

Now define \eqref{eq:sh' for one} to be the composition
$$ C_p(A) \otimes C_q(C) \stackrel{i_A \otimes i_C}{\longrightarrow}
C_p(A\otimes C) \otimes C_q(A \otimes C)
\stackrel{\operatorname{sh}'}{\longrightarrow} C_{p+q+2}(A \otimes C)
$$

In the graded case, the sign rule is as follows: any $a_i$ has parity
$|a_i|+1$, and similarly for $c_i$.

\begin{thm} \label{thm: exterior product cyclic}
The map ${\bf{\operatorname{sh}}} + u{\bf{\operatorname{sh}'}}$ defines a
$k[[u]]$-linear, $(u)$-adically continuous quasi-isomorphism
$$(C_{\bullet}(A) \otimes C_{\bullet}(C))[[u]] \rightarrow CC_{\bullet}^-(A\otimes
C)$$
as well as
$$(C_{\bullet}(A) \otimes C_{\bullet}(C))((u))\rightarrow
CC_{\bullet}^{\operatorname{per}}(A\otimes C)$$
$$(C_{\bullet}(A) \otimes C_{\bullet}(C))((u))/u(C_{\bullet}(A)
\otimes C_{\bullet}(C))[[u]] \rightarrow CC_{\bullet}(A\otimes C)$$
(differentials on the left hand sides are equal to  $b\otimes
1 + 1 \otimes b + u(B\otimes 1 + 1 \otimes B$)).
\end{thm}
Note that the left hand side of the last formula maps to the
tensor product $CC_{\bullet}(A) \otimes CC_{\bullet}(C):$
$\Delta(u^{-p}c\otimes c') = (u^{-1}\otimes 1 + 1 \otimes
u^{-1})^pc\otimes c'$. One checks that this map is an embedding
whose cokernel is the kernel of the map $u \otimes 1 - 1 \otimes
u$, or $S \otimes 1 - 1 \otimes S$ where $S$ is as in
(\ref{ses:CC-C1}). From this we get
\begin{thm} \label{thm:kunneth}
There is a long exact sequence
$$ \rightarrow HC_n(A\otimes C) \stackrel{\Delta}{\longrightarrow} \bigoplus_{p+q =
n} HC_p (A) \otimes HC_q (C)
\stackrel{S\otimes 1 - 1\otimes S}{\longrightarrow}$$
$$ \bigoplus_{p+q = n-2} HC_p (A) \otimes HC_q (C)
\stackrel{\times}{\longrightarrow} HC_{n-1}(A \otimes C)
\stackrel{\Delta}{\longrightarrow} $$
\end{thm}
\subsection{Pairings between chains and cochains} \label{pairings}
Let us start with a motivation for what follows. We will see below that, when the ring of functions on a manifold is replaced by an arbitrary algebra, then Hochschild chains play the role of differential forms (with the differential $B$ replacing the de Rham differential) and Hochschild cochains play the role of multivector fields. We are looking for an analog of pairings that are defined in the classical context, namely the contraction of a form by a multivector field and the Lie derivative. In classical geometry, those pairings satisfy various algebraic relations that we try to reproduce in general. We will show that these relations are true up to homotopy; a much more complicated question whether they are true up to all higher homotopies is postponed until section \ref{ncdc}.
For a graded algebra $A$, for $D \in C^d(A,A)$, define
\begin{equation} \label{eq: i}
i_D(a_0 \otimes \ldots \otimes a_n)= (-1)^{|D|\sum_{i\leq d}(|a_i| + 1)}a_0
D(a_1, \ldots, a_d) \otimes a_{d+1} \otimes \ldots \otimes a_n
\end{equation}
\begin{proposition}  \label{prop: properties of i}
$$ [b,i_D] = i_{\delta D}$$
$$i_D i_E = (-1)^{|D||E|} i_{E \smile D} $$
\end{proposition}
Now, put
\begin{equation} \label{eq: L}
L_D(a_0 \otimes \ldots \otimes a_n)=\sum _{k=1}^{n-d} \epsilon _k a_0
\otimes \ldots \otimes D(a_{k+1}, \ldots, a_{k+d}) \otimes \ldots \otimes
a_n +
\end{equation}
$$ \sum _{k=n+1 -d}^{n} \eta _k D (a_{k+1}, \ldots, a_n, a_0, \ldots )
\otimes \ldots \otimes a_k
$$

(The second sum in the above formula is taken over all cyclic permutations
such that $a_0$ is inside $D$). The signs are given by
$$ \epsilon _k = (|D| + 1)\sum _{i=0}^{k} (|a_i| +1)$$
and
$$
\eta _k = |D|+ 1 + \sum_{i \leq k}(|a_i|+1)\sum_{i \geq k}(|a_i|+1)
$$
\begin{proposition}
$$[L_D, L_E]=L_{[D,E]}$$
$$[b, L_D] + L_{\delta D} = 0$$
$$[L_D, B] = 0$$
\end{proposition}
Now let us extend the above operations to the cyclic complex. Define
\begin{equation} \label{eq:S_D}
S_D(a_0 \otimes \ldots \otimes a_n)= \sum_ {j\geq 0;\; k\geq j+d}\epsilon
_{jk} 1  \otimes a_{k+1} \otimes \ldots a_0 \otimes \ldots \otimes
\end{equation}
$$D(a_{j+1}, \ldots, a_{j+d}) \otimes
\ldots \otimes a_k$$ (The sum is taken over all cyclic
permutations; $a_0$ appears to the left of $D$). The signs are as
follows:
$$ \epsilon _{jk} = |D|(|a_0| + \sum_{i=1}^{n} (|a_i|+1)) + (|D|+1)\sum
_{j+1}^{k}(|a_i|+1)+\sum_{i \leq k}(|a_i|+1)\sum_{i \geq k}(|a_i|+1)$$

\begin{proposition} \label{prop: reinhart}
(\cite{R})
$$[b+uB, i_D + uS_D] - i_{\delta D} - uS_{\delta D} = L_D $$
\end{proposition}
\begin{proposition} \label{prop: gdt}
(\cite{DGT})
There exists a linear transformation $T(D,E)$ of the Hochschild chain
complex, bilinear in $D, \;E \in C^{\bullet}(A,A)$, such that
$$[b+uB, T(D,E)] - T(\delta D, E) - (-1)^{|D|} T(D, \delta E) =$$

$$=[L_D , i_E + u S_E] - (-1)^{|D|+ 1}(i_{[D,E]}+uS_{[D,E]})$$
\end{proposition}
\subsection{Hochschild and cyclic complexes of $A_{\infty}$ algebras}\label{ss:Complexes of A infty algebras} They are defined exactly as for DG algebras, the chain differential $b$ being replaced by $L_m$ and the cochain differential $\delta$ by $[m,?]$ where $m$ is the Hochschild cochain from the definition of an $A_\infty$ algebra.
\subsection{Rigidity of periodic cyclic homology} \label{ss:rigid}
The following is the Goodwillie rigidity theorem \cite{Good}. A proof using operations on Hochschild and cyclic complexes is given in \cite{NT2}.
Let $A$ be an associative algebra over a ring $k$ of
characteristic zero. Let $I$ be a nilpotent two-sided ideal of
$A$. Denote $A_0= A/I$.
\begin{thm} \label{thm:rigid} (Goodwillie)
The natural map $CC_{\bullet}^{\operatorname{per}}(A) \rightarrow
CC_{\bullet}^{\operatorname{per}}(A/I)$ 
is a quasi-isomorphism.
\end{thm}
\subsection{Smooth functions} \label{ss:cinfty}
For a smooth manifold $M$ one can compute the Hochschild and
cyclic homology of the algebra $C^{\infty}(M)$ where the tensor
product in the definition of the Hochschild complex is one of the
following three:
\begin{equation} \label{eq:tensor1}
C^{\infty}(M)^{\otimes n} = C^{\infty}(M^n);
\end{equation}
\begin{equation} \label{eq:tensor2}
C^{\infty}(M)^{\otimes n} = \operatorname{germs}_{\Delta}C^{\infty}(M^n);
\end{equation}
\begin{equation} \label{eq:tensor3}
C^{\infty}(M)^{\otimes n} = \operatorname{jets}_{\Delta}C^{\infty}(M^n)
\end{equation}
where $\Delta$ is the diagonal.
\begin{thm} \label{thm:HKRsmooth} The map
$$\mu: f_0 \otimes f_1 \otimes \ldots \otimes f_n \mapsto \frac{1}{n!}
f_0df_1 \ldots df_n$$
defines a quasi-isomorphism of complexes
$$C_{\bullet}(C^{\infty}(M)) \rightarrow (\Omega^{\bullet}(M), 0)$$
and a ${\mathbb{C}}[[u]]$-linear, $(u)$-adically continuous
quasi-isomorphism
$$CC^-_{\bullet}(C^{\infty}(M)) \rightarrow (\Omega^{\bullet}(M)[[u]], ud)$$
Localizing with respect to $u$, we also get quasi-isomorphisms
$$CC_{\bullet}(C^{\infty}(M)) \rightarrow
(\Omega^{\bullet}(M)[u^{-1},u]]/u\Omega^{\bullet}(M)[[u]], ud)$$
$$CC^{\operatorname{per}}_{\bullet}(C^{\infty}(M)) \rightarrow
(\Omega^{\bullet} (M)[u^{-1},u]], ud)$$
\end{thm}

This theorem, for the tensor products (\ref{eq:tensor2},
\ref{eq:tensor3}), is due essentially to Hochschild, Kostant and
Rosenberg (the Hochschild case) and to Connes (the cyclic cases).
For the tensor product (\ref{eq:tensor1}), see \cite{Te}.
\subsubsection{Holomorphic functions} \label{sss:holofunctions}
Let $M$ be a complex manifold with the structure sheaf $\cal{O}_M$ and the
sheaf of holomorphic forms $\Omega^ {\bullet}_M$.
If one uses one of the following definitions of the tensor product, then
$C_{\bullet}({\cal{O}}_M)$, etc. are complexes of sheaves:
\begin{equation} \label{eq:tensor2holo}
{\cal{O}}_M^{\otimes n} = \operatorname{germs}_{\Delta}{\cal{O}}_{M^n};
\end{equation}
\begin{equation} \label{eq:tensor3holo}
{\cal{O}}_M^{\otimes n} = \operatorname{jets}_{\Delta}{\cal{O}}_{M^n}
\end{equation}
where $\Delta$ is the diagonal.
\begin{thm} \label{thm:holo} The map
$$\mu: f_0 \otimes f_1 \otimes \ldots \otimes f_n \mapsto \frac{1}{n!}
f_0df_1 \ldots df_n$$
defines a quasi-isomorphism of complexes of sheaves
$$C_{\bullet}({\cal{O}}_M) \rightarrow (\Omega^{\bullet} _M, 0)$$
and a ${\mathbb{C}}[[u]]$-linear, $(u)$-adically quasi-isomorphism of
complexes of sheaves
$$CC^-_{\bullet}({\cal{O}}_M) \rightarrow (\Omega^{\bullet} _M[[u]], ud)$$
\end{thm}
Similarly for the complexes $CC_{\bullet}$ and $CC^{\operatorname{per}}$.
\section{Operads}\label{s:operads}
\subsection{Definition and basic properties}\label{ss:def ops}
\begin{definition}\label{dfn:operad}
An operad $\cP$ in a symmetric monoidal category with direct sums and products ${\mathcal C}$ is:

a) a collection of objects $\cP(n),$ $n\geq 1,$ with an action of the symmetric group $\Sigma_n$ on $\cP(n)$ for every $n$;

b) morphisms
$${\rm{op}}_{n_1, \ldots, n_k}: \cP(k)\otimes  \cP({n_1})\otimes\ldots \otimes \cP({n_k})\to \cP(n_1+\ldots + n_k)$$
such that:

(i) $$\bigoplus_{\sigma\in \Sigma_k} {\rm{op}}_{n_{\sigma(1)}, \ldots, n_{\sigma(k)}}: \bigoplus _{\sigma\in \Sigma_k} \cP(k)\otimes\cP(n_{\sigma(1)})\otimes \ldots\\otimes cP(n_{\sigma(k)})\to \cP(n_1+\ldots +n_k)$$
is invariant under the action of the cross product $\Sigma_k\ltimes (\Sigma_{n_1}\times\ldots\times \Sigma_{n_k});$

(ii) the diagram
$$
\begin{CD}
\cP(k)\otimes \bigotimes_i \cP(l_i) \otimes\bigotimes_{i,j} \cP(m_{i,j}) @>>> \cP(k)\otimes \bigotimes _i\cP(\sum_jm_{i,j})\\
@VVV                                                             @VVV\\
\cP(\sum_i l_i)\otimes \bigotimes_{i,j} \cP(m_{i,j})   @>>>\cP(\sum_{i,j}m_{i,j})
\end{CD}
$$
is commutative.
\end{definition}
Here is an equivalent definition: an operad is an object $\cP(I)$ for any nonempty finite set $I$, functorial with respect to bijections of finite sets,
together with a morphism
$${\rm{op}}_f: \cP(f)\to \cP(I)$$
for every surjective map $f:I\to J,$ where we put
$$\cP(f)=\cP(J)\otimes\bigotimes _{j\in J} \cP(f^{-1}(\{j\});$$
for every pair of surjections $I \stackrel{g}{\rightarrow} J\stackrel{f}{\rightarrow} K,$ and any element $k$ of $K$, set
$$g_k =g|(fg)^{-1}(\{k\}): (fg)^{-1}(\{k\})\to g^{-1}(\{k\}).$$
We require the diagram
\begin{equation}\label{eq:assoc of P(I)}
\begin{CD}
\cP(K)\otimes \bigotimes _{k\in K} \cP(g_k) @>>> \cP(fg)\\
@VVV                                                                        @VVV\\
 \cP(g)@>>> \cP(I)
\end{CD}
\end{equation}
to be commutative.

It is easy to see that the two definitions are equivalent. Indeed, starting from Definition \ref{dfn:operad}, put
$$\cP(I)=\bigoplus_{\phi:\{1,\ldots, k\}\isomoto I} \cP(k)/\sim$$
where $(\psi, p)\eq (\phi, \phi\psi^{-1}p).$ In the opposite direction, define $\cP(k)=\cP(\{1,\ldots, k\}).$

An element $e$ of $\cP(1)$ is a unit of $\cP$ if ${\rm {op}}_1(p,e)=p$ for all $p\in \cP(1),$ ${\rm {op}}_n(e,p)=p$ for all $p\in \cP(n)$ for the operation ${\rm {op}}_n:\cP(1)\otimes \cP(n)\to \cP(n)$. (This definition works for categories such as spaces, complexes, etc.; in general, instead of an object $e$, one should talk about a morphism from the object ${\bf 1}$ to $\cP(1)$). An operad is unital if it has a unit. For a unital operad $\cP,$ and for every map, surjective or not, morphisms
\begin{equation}\label{eq:tilde etc}
{\rm{op}}_f: \cP(f)\to \cP(\widetilde{I}), \;\widetilde{I}_f=I\coprod (J-f(I)),
\end{equation}
can be defined  by mapping ${\bf 1}$ to $\cP$ using the unit, and then constructing the operation ${\rm{op}}_{\tilde{f}},$ $\overline{f}(i)=f(i)$ for $i\in I,$ $\overline{f}(j)=j$ for $j \in J.$ In particular, taking $f$ to be a map whose image consists of one point, we get morphisms $\circ_i: \cP(k)\otimes \cP(n)\to \cP(n+k-1)$ for $1\leq k\leq n.$
\begin{remark}\label{remark: I tilde etc}
We can define an operad $\cP$ as a collection $\cP(n)$ with actions of $\Sigma_n$ and with products ${\rm{op}}_f$ as in \eqref{eq:tilde etc} for any map $f:I\to J,$ surjective or not, subject to the condition of invariance under $\Sigma_n$ associative in the following sense. For maps $I\stackrel{f}{\rightarrow}J{\stackrel{f}{\rightarrow}}K,$ define ${\tilde{g}}:I\to {\widetilde{J}}_f$ as the composition $I{\stackrel{g}{\rightarrow}}J\rightarrow {\widetilde{J}}_f,$ and ${\widetilde{fg}}:{\widetilde{I}}_f\to K$ as $fg$ on $I$ and $f$ on $J-f(I).$ Observe that
$$\cP(K)\otimes \bigotimes _{k\in K}\cP(f^{-1}(\{k\}))\otimes _{j\in J}\cP(g^{-1}(\{j\}))\isomoto \cP(K)\otimes \bigotimes_{k\in K}\cP(g_k);$$
$$\cP({\widetilde{J}}_f)\otimes\bigotimes_{j\in J}\cP(g^{-1}(\{j\}))\isomoto \cP({\widetilde{g}});$$
$$\cP(K)\otimes \bigotimes_{k\in K}\cP({\widetilde{f^{-1}(\{k\})}}_{g_k})\isomoto \cP({\widetilde{fg}});$$
we get the diagram
\begin{equation}\label{eq:assoc of P(I) tilde etc}
\begin{CD}
\cP(K)\otimes \bigotimes _{k\in K} \cP(g_k) @>>> \cP({\widetilde{fg}})\\
@VVV                                                                        @VVV\\
 \cP({\widetilde{g}})@>>> \cP({\widetilde{I}}_{fg})
\end{CD}
\end{equation}
that is required to be commutative. We can take this for the definition of an operad. Any unital operad is an example, but there are others which are not exactly unital.
\end{remark}
\begin{example}\label{ex:operad End}
For an object $A,$ put ${\rm{End}}_A(n)={\rm{Hom}}(A^{\otimes n}, A)$. The action of $\Sigma_n$ and the operations ${\rm{op}}$ are the obvious ones. This is the operad of endomorphisms of $A$.
\end{example}
A morphism of operads $\cP\to \cQ$ is a collection of morphisms $\cP(n)\to \cQ(n)$ that agree with the action of $\Sigma_n$ and with the operations ${\rm{op}}_{n_1, \ldots, n_k}.$ A morphism of unital operads is a morphism that sends the unit of $\cP$ to the unit of $\cQ.$
\subsubsection{Algebras over operads}\label{sss:algebras over operads}
An algebra over an operad $\cP$ is an object $A$ with a morphism $\cP\to {\rm{End}}_A.$ In other words, an algebra over $\cP$ is an object $A$ together with $\Sigma_n$-invariant morphisms
$$\cP(n)\otimes A^{\otimes n}\to A$$
such that the diagram
$$
\begin{CD}
\cP(k)\otimes \bigotimes _{i=1}^{k}\cP(n_i)\otimes A^{\otimes\sum_{i=1}^{k} n_i} @>>> \cP(\sum_{i=1}^{k} n_i)\otimes A^{\otimes\sum_{i=1}^{k} n_i}\\
@VVV                                                                                 @VVV\\
\cP(k)\otimes A^{\otimes k}@>>>A
\end{CD}
$$
is commutative.
For an algebra over a unital operad $\cP,$ one assumes in addition that the composition $A\isomoto{\bf 1}\otimes A\to \cP(1)\otimes A\to A$ is the identity.

A free algebra over $\cP$ generated by $V$ is
$${\rm{Free}}_{\cP}(V)=\bigoplus _n \cP(n)\otimes _{\Sigma_n} V^{\otimes n}$$
The action of $\cP$ combines the operadic products on $\cP$ and the free (tensor) product on $V^{\otimes \bullet}$. The free algebra satisfies the usual universal property: For any $\cP$-algebra $A,$ a morphism of objects $V\to A$ extends to a unique morphism of $\cP$-algebras ${\rm{Free}}_{\cP}(V)\to A.$
\subsubsection{Colored operads}\label{sss:Col op}
A colored operad is a set $X$ (whose elements are called colors), an object $\cP(x_1, \ldots, x_n; y)$ for every finite subset $\{x_1, \ldots, x_n\}$ and every element $y$ of $X$, an action of ${\rm{Aut}}(\{x_1, \ldots, x_n\})$ on $\cP(x_1, \ldots, x_n; y),$
and morphisms
$${\rm{op}}: \cP(y_1, \ldots, y_k; z)\otimes \bigotimes_{i=1}^k \cP(\{x_{ij}\}_{1 \leq j\leq n_i};y_i)\to \cP(\{x_{ij}\}_{1\leq i\leq k\; 1 \leq j\leq n_i}; z),$$ subject to the axioms of invariance and associativity generalizing the ones in Definition \ref{dfn:operad}. An algebra over a colored operad $\cP$ is a collection of objects $A_x$, $x\in X$, together with operations
$$\cP(x_1, \ldots, x_n; y)\otimes A_{x_1}\otimes \ldots \otimes A_{x_n}\to A_y,$$
subject to axioms of invariance and associativity.
\subsubsection{Topological operads}\label{sss:top operads} A topological operad is an operad in the category of topological spaces where $\otimes$ stands for the Cartesian product. If $\cP$ is a topological operad then $C_{-\bullet}(\cP)$ is an operad in the category of complexes. (We use the minus sign to keep all our complexes cohomological, i.e. with differential of degree $+1$). Its $nth$ term is the singular complex of the space $\cP(n).$
\subsection{DG operads}\label{ss:DG operads} A DG operad is an operad in the category of complexes. A DG operad for which $\cP(n)=0$ for $n\neq 1$ is the same as an associative DG algebra.
\subsection{Cofibrant DG operads and algebras} \label{ss:Cofibrant DG operads and algebras} A free DG operad  generated by a collection of complexes $V(n)$ with an action of $\Sigma(n)$ is defined as follows. Let ${\rm{FreeOp}}(V)(n)$ be the direct sum over isomorphism classes of rooted trees $T$ whose external vertices are labeled by indexes $1,\ldots, n$:
$${\rm{FreeOp}}(V)(n)=\bigoplus _T \bigotimes _{{\rm{Internal}}\; {\rm{vertices}}\;v \;{\rm {of}}\; T} V(\{{\rm{edges}}\;{\rm{outgoing}}\; {\rm {from}} \;v\})$$
The action of the symmetric group relabels the external vertices; the operadic products graft the root of the tree corresponding to the argument in ${\rm{FreeOp}}(V)(n_i)$ to the vertex labeled by the index $i$ of the tree corresponding to the factor in ${\rm{FreeOp}}(V)(k)$.
A free operad has the usual universal property: for a DG operad $\cP,$ a morphism of collections of $\Sigma_n$ modules $V(n)\to \cP(n)$ extends to a unique morphism of operads ${\rm{FreeOp}}(V)\to \cP.$
\subsubsection{Semifree operads and algebras}\label{semifree} An algebra over a DG operad $\cP$ is semifree if:

 (i) its underlying graded $k$-module is a free algebra generated by a graded $k$-module $V$ over the underlying graded operad of $\cP;$

(ii) there is a filtration on $V:$ $0=V_0\subset V_1\subset \ldots,$ $V=\cup _n V_n,$ such that the differential sends $V_n$ to the suboperad generated by $V_k, \; k<n.$

One defines a semifree DG operad exactly in the same way, denoting by $V$ a collection of $\Sigma_n$-modules.

A DG operad $R$ (resp. an algebra $R$ over a DG operad $\cP$) is cofibrant if it is a retract of a semifree DG operad (resp. algebra), i.e. if there is a semifree $Q$ and maps $R\stackrel{i}{\longrightarrow} Q\stackrel{j}{\longrightarrow}R$ such that $ji={\rm{id}}_R.$

We say that a morphism of DG operads (resp. of algebras over a DG operad) is a fibration if it is surjective. We say that a morphism is a weak equivalence if it is a quasi-isomorphism. It is easy to see that the above definition of a cofibrant object is equivalent to the usual one: for every morphism $p:P\to Q$ that is a fibration and a weak equivalence, and for every $f:R\to Q,$ there is a morphism ${\tilde{f}}:R\to P$ such that $p{\tilde{f}}=f.$

\subsubsection{Cofibrant resolutions} A cofibrant resolution of a DG operad $\cP$ is a cofibrant DG operad $\cR$ together with a surjective quasi-isomorphism of $DG$ operads $\cR\to \cP.$ Every DG operad has a cofibrant resolution. For two such resolutions $\cR_1$ and $\cR_2,$ there is a morphism $\cR_1\to \cR_2$ over $\cP.$ Any two such morphisms are homotopic in the following sense. Let $\Omega^\bullet([0,1])$ be the DG algebra $k[t,dt]$ with the differential sending $t$ to $dt$. Let ${\operatorname{ev}}_a: \Omega^\bullet([0,1])\to k$ be the morphism of algebras sending $t$ to $a$ and $dt$ to zero. Two morphisms $f_0, f_1:\cR_1\to \cR_2$ are homotopic if there is a morphism $f:\cR_1\to \cR_2\otimes \Omega^\bullet([0,1])$ such that ${\id}_{\cR_2}\otimes {\operatorname{ev}}_a=f_a$ for $f=0, 1.$
\subsection{Bar and cobar constructions}\label{ss: and cobar resolutions} The references for this subsection are \cite{GK} for the case of operads and \cite{GJ} for the case of DG operads.
\subsubsection{Cooperads and coalgebras} The definition a cooperad and a coalgebra over it is dual to that of an operad and an algebra over it. In particular, a cooperad is a collection of objects $\cB(n)$ with actions of $\Sigma_n$, together with morphisms
$$\cB(n_1+\ldots+n_k)\to \cB(k)\otimes\cB(n_1)\otimes \ldots \otimes \cB(n_k) ,
$$
and a coalgebra C over $\cB$ is an object $C$  together with morphisms
$$C\to \cB(n)\otimes C^{\otimes n},$$
subject to the conditions of $\Sigma_n$-invariance and coassociativity.
A cofree coalgebra over $\cB$ (co)generated by a complex $W$ is defined as
$${\rm{Cofree}}_{\cB}(W)=\prod _{n\geq 1}(\cB(n)\otimes W^{\otimes n})^{\Sigma_n};$$
a cofree cooperad (co)generated by a collection of $\Sigma_n$-modules $W=\{W(n)\}$ is by definition
$${\rm{CofreeCoop}}(W)(n)=\prod _T \bigotimes _{{\rm{Interior}}\; {\rm{vertices}}\;v \;{\rm {of}}\; T} W(\{{\rm{edges}}\;{\rm{outgoing}}\; {\rm {from}} \;v\})$$
The cooperadic coproducts are induced by cutting a tree in all possible ways into a subtree containing the root and $k$ subtrees $T_1,\;\ldots,T_k,$ such that the external vertices of $T_i$ are exactly the external vertices of $T$ labeled by $n_1+\ldots +n_{i-1}+1, \ldots, n_1+\ldots +n_{i}.$ The coaction of $\cB$ on the cofree coalgebra is a combination of the cooperadic coproducts on $\cB$ and the cofree coproduct on the tensor coalgebra $W^{\otimes\bullet}.$
\subsubsection{The bar construction}\label{sss:bar}
Let $\cP$ be a DG operad as in Remark \ref{remark: I tilde etc}. The bar construction of $\cP$ is the cofree DG cooperad ${\rm{CofreeCoop}}(\cP[-1])$ with the differential defined by
$d=d_1+d_2$ where, for a rooted tree $T,$
$$d_1 ({\otimes _{{\rm{Internal}}\; {\rm{vertices}}\;v \;{\rm {of}}\; T}}(p(v)))=\sum \pm \otimes_{v'\neq v}p(v')\otimes d_\cP p(v),$$
$p(v)\in \cP(\{ {\rm{edges}}\;{\rm{outgoing}}\;{\rm {from}} \;v\})[1],$ where $d_\cP$ is the differential on $\cP[1];$
$$d_2({\otimes}_v(p(v))) =\sum _{{\rm{Internal}}\; {\rm{edges}}\;e \;{\rm {of}}\; T}\pm{\bf c}(e)(\otimes_v(p(v))).$$
Here ${\bf c}(e)$ is the operator of contracting the edge $e$ that acts as follows. Let $v_1$ and $v_2$ be vertices adjacent to $e$, $v_1$ closer to the root than $v_2.$ Let $T_e$ be the tree obtained from $T$ by contracting the edge $e$. Consider the operation
$$\cP(\{\; {\rm{edges}}\;{\rm{of}}\;T\;{\rm{outgoing}}\;{\rm {from}} \;v_1\})\otimes \cP(\{\; {\rm{edges}}\;{\rm{of}}\;T\;{\rm{outgoing}}\;{\rm {from}} \;v_2\})$$
$$\stackrel{{\rm{op}}_{f_e}}{\rightarrow} \cP(\{ {\rm{edges}}\;{\rm{of}}\;T_e\;{\rm{outgoing}}\;{\rm {from}} \;v_1\})$$
corresponding to the map
$$f_e:\{\; {\rm{edges}}\;{\rm{of}}\;T\;{\rm{outgoing}}\;{\rm {from}} \;v_2\}\to \{ {\rm{edges}}\;{\rm{of}}\;T\;{\rm{outgoing}}\;{\rm {from}} \;v_1\}$$
sending all edges to $e$. The operator ${\bf c}(e)$ replaces $T$ by $T_e$ and the tensor factor $p(v_1)\otimes p(v_2)$ by its image under ${\rm{op}}_{f_e}.$ The signs both in $d_1$ and $d_2$ are computed according to the following rule: start from the root of $T$ and advance to the vertex, resp. to the edge. Passage through every factor $p(v)$ at a vertex $v$ introduces the factor $(-1)^{|p(v|}$ (the degree in $\cP[1]$). 

It is easy to see that this differential defines a DG cooperad structure on ${\rm{CofreeCoop}}(\cP[-1]).$ We call this DG cooperad the bar construction of $\cP$ and denote it by ${\Br}(\cP).$

The dual definition starts with a DG cooperad $\cB$ and produces the DG operad ${\Cbr}(\cB).$
\begin{lemma} \label{lemma:bar of free operad}
Let $V=\{V(n)\}$ be a collection of $\Sigma_n$-modules. The embedding of $V$ into ${\rm {Bar}}{\rm{FreeOp}}(V)$ that sends an element of $V(n)$ into itself attached to a corolla with $n\,$ external vertices is a quasi-isomorphism of complexes.
\end{lemma}
Let $\cP$ be a DG operad as in Remark \ref{remark: I tilde etc}. Consider the map $\Cbr\Br (\cP)\to \cP$ defined as follows. A free generator which is an element of ${\rm{CofreeCoop}}(\cP[1])[-1]$ corresponding to a tree $T$ is sent to zero unless $T$ is a corolla, in which case it is sent to the corresponding element of $\cP(n).$
\begin{prop}\label{prop:bar cobar}
The above map $\Cbr\Br (\cP)\to \cP$ is a surjective quasi-isomorphism of DG operads.
\end{prop}
The DG operad $\Cbr\Br (\cP)$ is the standard cofibrant resolution of $\cP$.

\subsection{ Koszul operads}\label{ss: Koszul operads} The reference for this subsection is \cite{GK}. We give a very brief sketch of the main definitions and results. Let $V(2)$ be a $k$-module with an action of $\Sigma_2$. A quadratic operad generated by $V(2)$ is a quotient of the free operad $\FreeOp(\{V(2)\})$ by the ideal generated by a subspace $R$ of $(\FreeOp(\{V(2)\}))(3).$ 

For a $k$-module $X$, let $X^*=\Hom_k(X, k).$ Let $V(2)$ and $S$ be free $k$-modules of finite rank. The Koszul dual operad to a quadratic operad $\cP$ generated by $V(2)$ with relations $R$ is the quadratic operad $\cP^\vee$ generated by $V(2)[1]^*$ subject to the orthogonal complement $R^\perp$ to $R$.

By definition, $({\mathcal P}^\vee)^\vee={\mathcal P}.$ There is a natural morphism of operads $\cP^\vee\to \Br(\cP)^*$. The quadratic operad $\cP$ is {\em Koszul} if this map is a quasi-isomorphism. 

A quadratic operad $\cP$ is Koszul if and only if $\cP$ is.

The above constructions may be carried out if $V(2)$ is replaced by a pair $(V(1), V(2)).$

For a Koszul operad $\cP,$ the DG operad $\Cbr(\cP^\vee)$ is a cofibrant resolution of $\cP.$ We will denote it by $\cP_\infty.$

\subsection{Operads $\rm{As}$, $\rm{Com}$, $\rm{Lie}$, $\rm{Gerst}$, $\rm{Calc}$, $\rm{BV}$, and their $\infty$ analogs}\label{ss:examples of operads}
\subsubsection{$\rm{As}$, $\rm{Com}$, and $\rm{Lie}$}\label{sss:ass, comm, lie} Algebras over them are, respectively, graded associative algebras, graded commutative algebras, and graded Lie algebras.
\subsubsection{Gerstenhaber algebras}  \label{ss:gerstenhaber}
Let $k$ be the ground ring of characteristic zero.
A {\it Gerstenhaber algebra} is a graded space ${\mathcal{A}} $
together with
\begin{itemize}
\item A graded commutative associative algebra structure on
${\mathcal{A}}$;
\item a graded Lie algebra structure on ${\mathcal{A}}^{{\bullet}+1}$
such that
$$[a,bc]=[a,b]c+(-1)^{(|a| -1) |b|)}b[a,c]$$
\end{itemize}

\begin{example} \label{thm:gerstenhaber} Let $M$ be a smooth manifold. Then
$$ {\mathcal{V}}^{\bullet}_M = \wedge ^{\bullet} T_M$$
is a sheaf of Gerstenhaber algebras.
\end{example}
The product is the exterior product, and the bracket is the Schouten
bracket.
We denote by ${\mathcal V}(M)$ the Gerstenhaber algebra of global sections of this
sheaf.
\begin{example} \label{thm:gerstenhaber-2} Let $\g$ be a Lie algebra. Then
$$ C_{\bullet}(\g) = \wedge ^{\bullet} \g$$
is a Gerstenhaber algebra.
\end{example}
The product is the exterior product, and the bracket is the unique bracket
which turns $C_{\bullet}(\g)$ into a Gerstenhaber algebra and which is the
Lie bracket on $\g = \wedge ^1 (\g)$.

\subsubsection{Calculi} \label{ss:ncdc}
\begin{definition} \label{dfn:precalc}
A {\it precalculus} is a pair of a Gerstenhaber algebra
${\mathcal{V}}^{\bullet}$ and a graded space $\Omega ^{\bullet}$  together with
\begin{itemize}
\item a structure of a graded module over the graded commutative  algebra
${\mathcal{V}}^{\bullet}$ on $\Omega ^{-{\bullet}} $ (the corresponding action is
denoted by $i_a,\; a \in {\mathcal{V}}^{\bullet}$);
\item a structure of a graded module over the graded Lie  algebra
${\mathcal{V}}^{\bullet +1}$ on ${\Omega} ^{-{\bullet}}$ (the corresponding action
is denoted by $L_a,\; a \in {\mathcal{V}}^{\bullet}$)
such that
$$[L_a,i_b]=i_{[a,b]}$$
and
$$L_{ab} = (-1)^{|b|}L_a i_b +  i_a L_b$$
\end{itemize}
\end{definition}
\begin{definition} \label{dfn:calc}
A {\it calculus} is a precalculus
together with an operator $d$ of degree 1 on $\Omega ^{{\bullet}}$ such that
$d^2 = 0$ and
$$ [d,i_a]=(-1)^{|a|-1}L_a. $$
\end{definition}
\begin{example} \label{ex:calc-M}
For any manifold one defines a calculus $\Ca (M)$ with
${\mathcal{V}}^{\bullet}$ being the algebra of multivector fields,
$\Omega ^{\bullet}$ the space of differential forms, and $d$ the
de Rham differential. The operator $i_a$ is the contraction of a
form by a multivector field.
\end{example}
\begin{example} \label{ex:calc-0}
For any associative algebra $A$ one defines a calculus $\Ca _0 (A)$ by
putting
${\mathcal{V}}^{\bullet} = H^{\bullet}(A,A)$ and $\Omega^{\bullet} = H_{\bullet}
(A,A)$. The five operations from Definition \ref{dfn:calc} are the cup
product, the Gerstenhaber bracket, the pairings $i_D$ and $L_D$, and the
differential $B$, as in \ref{pairings}. The fact that it is indeed a
calculus follows from Theorem \ref{prop: gdt}.
\end{example}
A differential graded (dg) calculus is a calculus with extra differentials
$\delta$ of degree 1 on ${\mathcal{V}}^{\bullet}$ and $b$ of degree $-1$ on
$\Omega ^{\bullet}$ which are derivations with respect to all the
structures.
\begin{definition}\label{dfn:calc h,u}
1) An $\hbar$-calculus is a precalculus over the algebra $k[\hbar],$ $|\hbar|=0,$ together with a $k[\hbar]$-linear operator of degree $+1$ on $\Omega^{-\bullet}$ satisfying
$$d^2=0;\;[d,\iota_a]=(-1)^{|a|-1}\hbar L_a$$

2) A $u$-calculus is a precalculus over the algebra $k[u],$ $|u|=2,$ together with a $k[u]$-linear operator of degree $-1$ on $\Omega^{-\bullet}$ satisfying
$$d^2=0;\;[d,\iota_a]=(-1)^{|a|-1}u L_a$$
\end{definition}
\subsubsection{BV algebras}\label{sss:BV algebras}
\begin{definition}\label{dfn:BV}
A Batalin-Vilkovisky (BV) algebra is a Gerstenhaber algebra together with an operator $\Delta: {\mathcal{A}}\to {\mathcal{A}}$ of degree $-1$ satisfying
$$\Delta^2=0$$
and
\begin{equation}\label{eq:BV}
\Delta(ab)-\Delta(a)b-(-1)^{|a|}a\Delta(b)=(-1)^{|a|-1}[a,b]
\end{equation}
\end{definition}
Note that the above axioms imply
\begin{equation}\label{eq:BV 1}
\Delta([a,b])-[\Delta(a),b]+(-1)^{|a|-1}[a,\Delta(b)]=0
\end{equation}
There are two variations of this definition.
\begin{definition}\label{dfn:BV h,u}
1) A ${\rm{BV}}_\hbar$-algebra is a Gerstenhaber algebra over the algebra $k[\hbar],$ $|\hbar|=0,$ with a $k[\hbar]$-linear operator $\Delta: {\mathcal{A}}\to {\mathcal{A}}$ of degree $-1$ satisfying
$$\Delta^2=0,$$
the identity \eqref{eq:BV 1}, and
\begin{equation}\label{eq:BV h}
 \Delta(ab)-\Delta(a)b-(-1)^{|a|}a\Delta(b)=(-1)^{|a|-1}\hbar[a,b]
\end{equation}

2) 1) A ${\rm{BV}}_u$-algebra is a Gerstenhaber algebra over the algebra $k[u],$ $|u|=2,$ with a $k[u]$-linear operator $\Delta: {\mathcal{A}}\to {\mathcal{A}}$ of degree $+1$ satisfying
$$\Delta^2=0,$$
the identity \eqref{eq:BV 1}, and
\begin{equation}\label{eq:BV u}
\Delta(ab)-\Delta(a)b-(-1)^{|a|}a\Delta(b)=(-1)^{|a|-1}u[a,b]
\end{equation}
\end{definition}
\begin{proposition}\label{prop:koszul duals}
For a DG operad $\cP,$ denote by $\cP^\vee$ its Koszul dual.
\begin{enumerate}
\item
$\As ^\vee=\As;\; \Com ^\vee =\Lie;\; \Lie^\vee=\Com;$
\item
a complex $A$ is an algebra over $\Gerst ^\vee$ if and only if $A[1]$ is an algebra over $\Gerst;$
\item
a complex $A$ is an algebra over $\BV _u^\vee$ if and only if $A[1]$ is an algebra over $\BV _\hbar;$
\item
a complex $A$ is an algebra over $\BV _\hbar^\vee$ if and only if $A[1]$ is an algebra over $\BV _u;$
\item
a pair of complexes $(A,\Omega)$ is an algebra over $\Ca _u^\vee$ if and only if $(A[1], \Omega)$ is an algebra over $\Ca _\hbar;$
\item
a pair of complexes $(A,\Omega)$ is an algebra over $\BV _u^\vee$ if and only if $(A[1], \Omega)$ is an algebra over $\Ca _\hbar.$
\item
All the operads above are Koszul.
\end{enumerate}
\end{proposition}
The above result was proved in \cite{GK} for $ \As,$ $\Com$ and $\Lie; $ in \cite{GJ} for $\Gerst;$ and in \cite{GCTV} for $\BV.$

\subsection{The Boardman-Vogt construction}\label{ss:The Boardman-Vogt construction} For a topological operad $\cP,$ Boardman and Vogt constructed in \cite{BV} another topological operad $W\cP,$ together with a weak homotopy equivalence of topological operads $W\cP\isomoto \cP.$ (In fact $W\cP$ is a cofibrant replacement of $\cP$). The space $W\cP(n)$ consists of planar rooted trees $T$ with the following additional data:
\begin{enumerate}
\item
internal vertices of $T$ of valency $j+1$ are decorated by points of $\cP(j);$ 
\item external vertices of $T$ are decorated by numbers from $1$ to $n,$ so that the map sending a vertex to its label is a  bijection between the set of internal vertices and $\{1,\ldots, n\};$ 
\item internal edges of $T$ are decorated by numbers $0\leq r\leq 1.$ The label $r$ is called the length of the edge. 
\end{enumerate}
If the length of an edge of a tree is zero, this tree is equivalent to the tree obtained by contracting the edge, the label of the new vertex defined via operadic composition from the labels of the two vertices incident to $e.$
\subsection{Operads of little discs}\label{ss:Operads of little discs} Let $D$ be the standard $k$-disc $\{x\in \bbR^k|\; |x|\leq 1.\}$. For $1\leq i \leq n,$ denote by $D_i$ a copy of $D$. Let $\LD_k(n)$ be the space of embeddings 
\begin{equation}\label{eq:embedding discs}
\coprod _{i=1}^n D_i \to D
\end{equation}
 whose restriction to every component is affine Euclidean. The collection $\{\LD_k(n)\}$ is an operad in the category of topological spaces. The action of $S_n$ is induced from the action by permutations of the $n$ copies of $D.$ Operadic composition is as follows. For embeddings 
$$f\colon\coprod _{i=1}^m D_i \to D$$ 
and
$$f_i\colon\coprod_{j_i=1}^{n_i} D_{j_i} \to D_i,$$
the embedding 
\begin{equation}\label{eq:op prod discs}
{\operatorname{op}}_{n_1,\ldots, n_m}(f;f_1,\ldots, f_m)\colon\coprod _{i=1}^m \coprod _{j=1}^{n_i} D_j\to D
\end{equation}
acts on every component $D_{j_i}$ by the composition $f\circ f_i.$
\subsection{Fulton-MacPherson operads}\label{ss:Fulton-MacPherson operads} The spaces $\FM_k(n)$ were defined by Fulton and MacPherson in \cite{FM}. The operadic structure on them was defined in \cite{GJ} by Getzler and Jones.

For $k>0,$ let $\bbR^+\ltimes \bbR^k$ be the group of affine transformations of $\bbR^k$ generated by positive dilations and translations. Define the configuration spaces to be
\begin{equation}\label{eq:def conf}
\Conf _k(n)=\{(x_1, \ldots, x_n)| x_i\in \bbR ^k, \, x_i\neq x_j\}/ (\bbR^+\ltimes \bbR^k)
\end{equation}
There are compactifications $\FM_k(n)$ of $\Conf_k(n)$ that form an operad in the category of topological spaces for each $k>0.$ As an operad of sets, $\FM_k(n)$ is the free operad generated by the collection of sets $\Conf_k(n)$ with the action of $S_n$.  In fact there are continuous bijections 
\begin{equation}\label{bij for FM}
\FreeOp (\{\Conf_k(n)\})\to \FM_k(n)
\end{equation}
The spaces $\FM_k(n)$ are manifolds with corners. They can be defined explicitly as follows. Consider the functions
$\theta_{ij}: \Conf_k(n)\to S^{k-1}$ and $\rho_{ijk}:\Conf_k(n)\to \bbR$ by 
\begin{equation}\label{eq:coords on FM}
\theta_{ij}(x_1,\ldots, x_n)=\frac{x_i-x_j}{|x_i-x_j|};\; \delta_{ijk}(x_1,\ldots, x_n)=\frac{|x_i-x_j|}{|x_i-x_k|}
\end{equation}
The map 
\begin{equation}\label{eq:ro teta}
\Conf_k(n)\to (S^{k-1})^{\binom{n}{2}}\times[0,+\infty]^{\binom{n}{3}}
\end{equation}
defined by all $\theta_{ij},\,i<j,$ and $\delta_ {ijk},\,i<j<k,$ can be shown to be an embedding. The space $\FM_k(n)$ can be defined as the closure of the image of this embedding.

Kontsevich and Soibelman proved in \cite{KoSo} that the topological operads $\FM_k$ and $\LD_k$ are weakly homotopy equivalent. In fact there is a homotopy equivalence of topological operads 
\begin{equation}\label{eq:equiv of Salvatore}
W\LD_k \isomoto \FM_k.
\end{equation}
constructed by Salvatore in \cite{Salva}, Prop. 4.9.
\subsection{The operad of framed little discs}\label{ss:Operads of framed little discs} This operad constructed analogously to the operad ${rm{LD}}_2 .$ By definition, ${\rm{FLD}}_2(n)$ is the space of affine embeddings \ref{eq:embedding discs} together with points $a_i\in \partial D_i,$ $a\in \partial D.$ The operadic compositions consist of those for ${\rm{LD}}_2$ and of  rotating the discs $D_i$ so that the marked points on the boundaries come together.
\subsection{The colored operad of little discs and cylinders}\label{ss:The colored operad of little discs and cylinders} The colored operad $\LC$ has two colors that we denote by $\co$ and $\cha.$ All spaces $\LC(x_1,\ldots, x_n; y)$ are empty if more than one $x_i$ is equal to $\cha$ or if one $x_i$ is equal to $\cha$ and $y=\co.$ For $n\geq 0,$ let
$$\LC(n)\stackrel{\rm{def}}{=}\LC({\co, \ldots, \co; \co})= {\LD_2(n)}$$
and 
$$\LC(n,1)\stackrel{\rm{def}}{=}\LC(\co, \ldots, \co,\cha;\cha).$$
The spaces $\LD(n)$ form a suboperad of $\LC.$
For $r>0,$ let $C_r$ be the cylinder $S^1 \times [0,r].$ By definition, $\LC(n,1)$ is the space of data $(r,g)$ where 
$$g\colon\coprod _{i=1}^k D_i \to C_r$$
is an embedding such that $g|D_i$ is the composition 
\begin{equation}\label{eq:g i}
D_i\stackrel{\widetilde{g_i}}{\rightarrow}\R\times [0,r]\stackrel{\rm{pr}}{\rightarrow}S^1\times [0,1]
\end{equation}
of the projection with an affine Euclidean map ${\widetilde{g}}$. The action of $S_n$ on $\LC(n,1)$ is induced by permutations of the components $D_i.$ Let us define operadic compositions of two types. The first is 
$$ \LC (m,1)\times \LC(n_1)\times \ldots \times \LC(n_m)\to \LC(n_1+\ldots +n_m, 1);$$
it is defined exactly as the operadic composition in \eqref{eq:op prod discs}, with $D$ replaced by $C_r.$ The second is
\begin{equation}\label{eq:comp2 for cyls}
\LC(n,1)\times \LC(m,1)\to \LC(n+m,1)
\end{equation}
For ${\widetilde{g_i}}: D_i\to C_r, 1\leq i\leq n,$ and ${\widetilde{g'_j}}: D_j\to C_{r'}, 1\leq j\leq m,$ as in \eqref{eq:g i}, define
\begin{equation}\label{eq:eqeq}
{\widetilde {g ''}}\colon (\coprod _{i=1}^n D_i )\coprod( \coprod _{j=1}^m D_j)\to \R\times [0,r+r']
\end{equation}
that sends $z\in D_i, 1\leq i\leq n,$ to the image of ${\widetilde {g}}(z)$ under the map $\R\times [0,r]\to \R\times [0,r+r'],$ $(x,t)\mapsto (x,t),$ and  $z\in D_j, 1\leq i\leq m,$ to the image of ${\widetilde {g'}}(z)$ under the map $\R\times [0,r']\to \R\times [0,r+r'],$ $(x,t)\mapsto (x,t+r).$ Let $g={\rm{pr}}\circ {\widetilde g}$ and $g'={\rm{pr}}\circ {\widetilde g'}$ The composition \eqref{eq:eqeq} of $(g,r)$ and $g'r')$ is by definition $({\rm{pr}}\circ {\widetilde {g ''}}, r+r').$

All other nonempty spaces $\LC(x_1,\ldots, x_n;y),$ in other words spaces 
\newline
$\LC(\co, \ldots \co, \cha, \ldots, \co; \cha),$ together with the actions of symmetric groups and with operadic compositions, are uniquely determined by the above and by the axioms of colored operads.
\subsubsection{The colored operad of little discs and framed cylinders} \label{sss:The colored operad of little discs and framed cylinders}
The colored operad $\LfC$ is defined exactly as $\LC$ above, with the following modifications. First, by definition, $\LC(n,1)$ is the space of data $(r, x_0,x_1,g)$ where $r$ and $g$ are as above, $x_0\in S^1\times \{0\},$ and $x_1\in \R\times \{r\},$ factorized by the action of the circle by rotations on the factor $S^1.$ The composition 
\begin{equation}\label{eq:comp2 for fr cyls}
\LfC(n,1)\times \LfC(m,1)\to \LfC(n+m,1)
\end{equation}
is defined as follows: given $(r,g,x_0,x_1)$ and $(r',g',x'_0,x'_1),$ their composition is $(r+r', g'', x_0, x'_1+x_1-x'_0)$ where $g''={\rm{pr}}\circ{\widetilde{g''}}$ and ${\widetilde{g''}}$ is exactly as in 
\eqref{eq:eqeq}, with the only difference that it sends $z\in D_j, 1\leq i\leq m,$ to the image of ${\widetilde {g'}}(z)$ under the map $\R\times [0,r']\to \R\times [0,r+r'],$ $(x,t)\mapsto (x+x_1-x'_0,t+r).$ Note that $\LC(1)$ is contractible but $\LfC(1)$ is homotopy equivalent to $S^1.$
\subsubsection{The Fulton-MacPherson version of $\LC$ and of $\LfC$}\label{sss:FM Cyl} Note first that the colored operad $\LC$ can be alternatively defined as follows: the spaces $\LC(n)$ are as above; the spaces $\LC(n,1)$ are defined as subspaces of $\LD_2(n+1)$ consisting of those embeddings \eqref{eq:embedding discs} that map the center of $D_{n+1}$ to the center of $D$. The action of the symmetric groups and the operadic compositions are induced from those of $\LD_2.$ Similarly, define the two-colored operad $\FMC$ as follows. Put $\FMC(n)=\FM(n);$ define $\FMC(n,1)$ to be the subspace of $\FM(n+1)$ consisting of data 
$$(T, \{c_v\}|{v\in\{{\operatorname{external}}\; {\operatorname{vertices}}\;{\operatorname{of}}\;T}\})$$
such that:
\begin{enumerate}
\item
$T$ is a rooted tree;
\item
$c_v\in \Conf(\{{\rm{edges}}\;{\rm{outgoing}}\; {\rm {from}} \;v\});$
\item
Consider the path from the root of $T$ to the external vertex labeled by $n+1$ ({\em the trunk} of $T$). Let $e_0$ be the edge on this path that goes out of a vertex $v.$ Let $c_v=(x_e)$ where $e$ are all edges outgoing from $v.$ Then $x_{e_0}=0.$
\end{enumerate}
We leave to the reader to define the operadic compositions and the action of the symmetric groups, as well as the Fulton-MacPherson analog $\FMfC$ of the two-colored operad $\LfC.$ 
\begin{proposition}\label{prop:GJ for cyl}
The two-colored operads $\FMC$ and $\LC,$ resp. $\FMfC$ and $\LfC,$ are weakly equivalent.
\end{proposition}
\section{DG categories}\label{s:DG cats} The contents of this section are taken mostly from \cite{DrDG}, \cite{KellerDG}, and \cite{Toen DG lectures}.
\subsection{Definition and basic properties}\label{ss:dfn propts DG cats} A differential graded (DG) category $A$ over $k$ is a collection $\Ob(A)$ of elements called objects and of complexes $A(x,y)$ of $k$-modules for every $x,y\in\Ob(A),$ together with morphisms of complexes
\begin{equation}\label{eq:dg cat}
A(x,y)\otimes A(y,z)\to A(x,z), \, a\otimes  b\mapsto ab,
\end{equation}
and zero-cycles ${\bold{1}} _x \in A(x,x)$, such that \eqref{eq:dg cat} is associative and ${\bold{1}} _x a=a{\bold 1}_y=a$ for any $a\in A(x,y).$ For a DG category, its homotopy category is the $k$-linear category ${\operatorname{Ho}}(A)$ such that $\Ob({\operatorname{Ho}}(A))=\Ob(A)$ and ${\operatorname{Ho}}(A)(x,y)=H^0(A(x,y),$ with the units being the classes of ${\bold{1}}_x$ and the composition induced by  \eqref{eq:dg cat}. 

A DG functor $A\to B$ is a map $\Ob(A)\to \Ob(B),$ $x\mapsto Fx,$ and a collection of morphisms of complexes $F_{x,y}\colon A(x,y)\to B(Fx,Fy), x,y\in \Ob(A),$ which commutes with the composition \eqref{eq:dg cat} and such that $F_{x,x}({\bold{1}}_{x})={\bold{1}}_{Fx}$ for all $x.$

The opposite DG category of $A$ is defined by $\Ob(A^{\op})=\Ob(A)$, $A^{\op}(x,y)=A(y,x),$ the unit elements are the same as in $A$, and the composition \eqref{eq:dg cat} is the one from $A,$ composed with the transposition of tensor factors.

For two DG categories $A$ and $B$, the tensor product $A\otimes B$ is defined as follows: $\Ob(A\otimes B)=\Ob(A)\times \Ob(B);$ we denote the object $(x,y)$ by $x\otimes y;$ 
$$(A\otimes B)(x\otimes y,x'\otimes y')=A(x,y)\otimes B(x',y');$$
$(a\otimes b)(a'\otimes b')=(-1)^{|a'||b|}aa'\otimes bb'; {\bold{1}}_{x\otimes y}={\bold{1}}_x\otimes {\bold 1}_y.$
\subsection{Cofibrant DG categories}\label{ss:cof DG cats} Cofibrant DG categories are defined exactly following the general principle of \ref{ss:Cofibrant DG operads and algebras}. 
\subsection{Quasi-equivalences}\label{ss:QE} A quasi-equivalence \cite{Tabu} between DG categories $A$ and $B$ is a DG functor $F:A\to B$ such that a) $F$ induces an equivalence of homotopy categories and b) for any $x,y\in \Ob(A),$ $F_{x,y}\colon A(x,y)\to B(Fx,Fy)$ is a quasi-isomorphism.
\subsection{Drinfeld localization}\label{Dr locn} For a full DG subcategory $C$ of a DG category $A$, the localization of $A$ with respect to $C$ is obtained from $A$ as follows. Consider DG categories $k_C$ and ${\mathcal{N}}_C;$ $\Ob(k_C)=\Ob({\mathcal{N}}_C)=\Ob(C);$ $k_C(x,y)={\mathcal{N}}_C(x,y)=0$ if $x\neq y;$ $k_C(x,x)=k\cdot {\bold{1}}_x;$ ${\mathcal{N}}_C(x,x)$ is equal to the free algebra generated by one element $\epsilon_x$ of degree $-1$  satisfying $d\epsilon _x={\bold{1}}_x$ for all $x\in Ob(C).$ The localization of $A$ is the free product $A * _{k_C} {\mathcal{N}}_C.$ In other words, it is a DG category $\cA$ such that:
\begin{enumerate}
\item
$\Ob(\cA)=\Ob(x);$
\item
there is a DG functor $i:A\to \cA$ which is the identity on objects;
\item
for every $x\in \Ob(C),$ there is an element $\epsilon_x$ of degree $-1$  in $\cA(x,x)$ satisfying $d\epsilon _x={\bold{1}}_x;$ 
\item
for any other DG category $\cA '$ together with a DG functor $i': A\to \cA'$ and elements $\epsilon '_x$ as above, there is unique DG functor $f:\cA\to \cA'$ such that $i'=f\circ i$ and $\epsilon _x\mapsto \epsilon '_x.$
\end{enumerate}
One has
$$\cA(x,y)=\bigoplus_{n\geq 0}\bigoplus_{x_1,\ldots, x_n\in \Ob(C)} A(x,x_1)\epsilon_{x_1}A(x_1,x_2)\epsilon_{x_2}\ldots \epsilon_{x_n}A(x_n,y);$$
it is easy to define the composition and the differential explicitly.
\subsection{DG modules over DG categories}\label{DGmods} A DG module over a DG category $A$ is a collection of complexes of $k$-modules $M(x),$ $x\in \Ob(A),$ together with morphisms of complexes 
\begin{equation}\label{eq:dg mod}
A(x,y)\otimes M(y)\to A(x), \, a\otimes  m\mapsto am,
\end{equation}
which is compatible with the composition \eqref{eq:dg cat} and such that ${\bold{1}}_x m=m$ for all $x$ and all $m\in M(x).$ A DG bimodule over $A$ is a collection of complexes $M(x,y)$ together with morphisms of complexes 
\begin{equation}\label{eq:dg bimod}
A(x,y)\otimes M(y,z) \otimes A(z,w)\to M(x,w), \, a\otimes m \otimes  b\mapsto amb,
\end{equation}
that agrees with the composition in $A$ and such that ${\bold{1}}_x m {\bold{1}}_y = m$ for any $x,y,m.$ We put $am=am{\bold{1}}_z$ and $mb={\bold{1}}_x mb.$ A DG bimodule over $A$ is the same as a DG module over $A\otimes A^{\op}.$
\subsection{Bar and cobar constructions for DG categories}\label{ss:Bar cobar for DG cat} The bar construction of a DG category $A$  is a DG cocategory $\Br (A)$ with the same objects where
$${\Br}(A)(x,y)=\bigoplus_{n\geq 0}\bigoplus _{x_1,\ldots, x_n}A(x,x_1)[1]\otimes A(x_1,x_2)[1]\otimes \ldots \otimes A(x_n,x)[1]$$
with the differential
$$d=d_1+d_2;$$
$$d_1(a_1|\ldots | a_{n+1})=\sum_{i=1}^{n+1}\pm (a_1|\ldots|da_i|\ldots |a_{n+1});$$
$$d_2(a_1|\ldots | a_{n+1})=\sum_{i=1}^{n}\pm (a_1|\ldots|a_ia_{i+1}|\ldots |a_{n+1})$$
The signs are $(-1)^{\sum_{j<i}(|a_i|+1)+1}$ for the first sum and $(-1)^{\sum_{j\leq i}(|a_i|+1)}$ for the second. The comultiplication is given by
$$\Delta(a_1|\ldots|a_{n+1})=\sum_{i=0}^{n+1}(a_1|\ldots|a_i)\otimes (a_{i+1}|\ldots |a_{n+1})$$
Dually, for a DG cocategory $B$ one defines the DG category $\Cbr(B)$. The DG category $\Cbr\Br(A)$ is a cofibrant resolution of $A$.
\subsubsection{Units and counits}\label{sss:units and counits} It is convenient for us to work with DG (co)categories without (co)units. For example, this is the case for $\Br(A)$ and $\Cbr(B)$ (we sum, by definition, over all tensor products with at least one factor). Let $A^+$ be the (co)category $A$ with the (co)units added, i.e. $A^+(x,y)=A(x,y)$ for $x\neq y$ and $A^+(x,x)=A(x,x)\oplus k\id _x.$ If $A$ is a DG category then $A^+$ is an augmented DG category with units, i.e. there is a DG functor $\epsilon: A^+\to k_{\Ob (A)}$. The latter is the DG category with the same objects as $A$ and with $k_I(x,y)=0$ for $x\neq y,$ $k_I(x,x)=k.$ Dually, one defines the DG cocategory $k^{\Ob(B)}$ and the DG functor $\eta: k^{\Ob(B)} \to B^+$ for a DG cocategory $B$.

\subsubsection{Tensor products} \label{sss:tensor products} For DG (co)categories with (co)units, define $A\otimes B$ as follows: ${\rm{Ob}}(A\otimes B)={\rm{Ob}}(A)\times {\rm{Ob}}(B);$ $(A\otimes B)((x_1,y_1), (x_2, y_2))=A(x_1,y_1)\otimes B(x_2,y_2);$ the product is defined as $(a_1\otimes b_1)(a_2\otimes b_2)=(-1)^{|a_2||b_1|}a_1a_2\otimes b_1b_2$, and the coproduct in the dual way. This tensor product, when applied to two (co)augmented DG (co)categories with (co)units, is again a (co)augmented DG (co)category with (co)units: the (co)augmentation is given by $\epsilon\otimes \epsilon$, resp. $\eta\otimes \eta.$

\begin{definition}\label{dfn:tensor product without (co)units}
For DG categories $A$ and $B$ without units, put
$$A\otimes B=\Ker(\epsilon\otimes \epsilon: A^+\otimes B^+\to k_{\Ob (A)} \otimes k_{\Ob (B)}).$$
Dually, for For DG cocategories $A$ and $B$ without counits, put
$$A\otimes B=\Coker(\eta\otimes \eta:k^{\Ob (A)}\otimes k^{\Ob (B)} \to A^+\otimes B^+).$$
\end{definition}

One defines a morphism of DG cocategories
\begin{equation}\label{eq:prod shuffle}
\Br(A)\otimes \Br(B)\to\Br(A\otimes B)
\end{equation}
by the standard formula for the shuffle product
\begin{equation}\label{eq:shuffle}
(a_1|\ldots|a_m)(b_1|\ldots|b_n)=\sum\pm (\ldots |a_i|\ldots|b_j|\ldots)
\end{equation}
The sum is taken over all shuffle permutations of the $m+n$ symbols $a_1, \ldots, a_m, b_1,$ $\ldots, b_n)$, i.e. over all permutations that preserve the order of the $a_i$'s and the order of the $b_j$'s. The sign is computed as follows: a transposition of $a_i$ and $b_j$ introduces a factor $(-1)^{(|a_i|+1)(b_j|+1)}.$ Let us explain the meaning of the factors $a_i$ and $b_j$ in the formula. We assume $a_i\in A(x_{i-1}, x_i)$ and $b_j \in B(y_{j-1}, y_j)$ for $x_i\in {\rm{Ob}}(A)$ and $y_j\in {\rm{Ob}}(B),$ $0\leq i\leq m,$ $0\leq j\leq m.$ Consider a summand $(\ldots |a_i|b_j|b_{j+1}|\ldots|b_k|a_{i+1}|\ldots).$ In this summand, all $b_p,$ $j\leq p\leq k,$ are interpreted as ${\id_{x_i}}\otimes b_p\in (A\otimes B)((x_i, y_{p-1}), (x_i, y_p)).$ Similarly, in the summand $(\ldots |b_i|a_j|a_{j+1}|\ldots|a_k|a_{i+1}|\ldots),$ all $a_p,$ $j\leq p\leq k,$ are interpreted as $a_p\otimes {\id_{y_i}}\in (A\otimes B)((x_{p-1},y_i), (x_p,y_i)).$
Dually, one defines the morphism of DG cocategories
\begin{equation}\label{eq:coprod coshuffle}
\Cbr(A\otimes B)\to \Cbr(A)\otimes \Cbr(B)
\end{equation}

\subsection{$A_\infty$ categories}\label{ss:A infty cats} An $A_\infty$ category is a natural generalization of both a DG category and an $A_\infty$ algebra. We refer the reader, for example, to \cite{KoSo1}.
\subsubsection{DG category ${\bf C}^\bullet(A,B)$}\label{sss:DG cat C(A,B)} For two DG categories $A$ and $B$, define the DG category ${\bf C}^\bullet(A,B)$ as follows. Its objects are $A_\infty$ functors $f:A\to B$. Define the complex of morphisms as
$${\bf C}^\bullet(A,B)(f,g)=C^\bullet (A, _f B _g)$$
where $_f B _g$ is the complex $B$ viewed as an $A_\infty$ bimodule on which $A$ acts on the left via $f$ and on the right via $g$. The composition is defined by the cup product as in the formula \eqref{eq:Hochcochains3}.
\begin{remark}\label{rmk:why not all bimodules?} Every $A_\infty$ functor $f:A\to B$ defines an $A_\infty$ $(A,B)$-bimodule $_f B$, namely the complex $B$ on which $A$ acts on the left via $f$ and $B$ on the right in the standard way. If for example $f, \, g:A\to B$ are morphisms of algebras then $C^\bullet (A,_f B _g)$ computes ${\rm{Ext}}^\bullet _{A\otimes B^{\op}}(_fB, _gB).$ What we are going to construct below does not seem to extend literally to all ($A_\infty$) bimodules. This applies also to related constructions of the category of internal homomorphisms, such as in \cite{Keller} and \cite{Toen int hom}. One can overcome this by replacing $A$ by the category of $A$-modules, since every $(A, B)$-bimodule defines a functor between the categories of modules.
\end{remark}
\subsubsection{The bialgebra structure on $\Br (C^\bullet (A,A))$}\label{sss:bialg stree on C(A,A)}
Let us first recall the product on \hfill the \hfill bar \hfill construction \hfill $\Br(C^\bullet(A,A))$ \hfill where \hfill $C^\bullet (A,A)$ \hfill is \hfill the \hfill algebra \hfill of \newline Hochschild cochains of $A$ with coefficients in $A$ (cf. \cite{GJ}, \cite{GV}). For cochains $D_i$ and $E_j,$ define
$$(D_1|\ldots|D_m)\bullet (E_1|\ldots |E_n)=\sum\pm (\ldots |D_1\{\ldots\}| \ldots |D_m\{\ldots\}|\ldots)$$
Here the space denoted by $\ldots$ inside the braces contains $E_{j+1}, \ldots, E_k;$ outside the braces, it contains $E_{j+1}|\ldots|E_k.$ The factor $D_i\{E_{j+1}, \ldots, E_k\}$ is the brace operation as in \eqref{eq:Hochcochains15}. The sum is taken over all possible combinations for which the natural order of $E_j$'s is preserved. The signs are computed as follows: a transposition of $D_i$ and $E_j$ introduces a sign $(-1)^{(|D_i|+1)(|E_j|+1)}.$ In other words, the right hand side is the sum over all tensor products of $D_i\{E_{j+1}, \ldots, E_k\},$ $k\geq j,$ and $E_p$, so that the natural orders of $D_i$'s and of $E_j$'s are preserved. For example,
$$(D)\bullet (E)=(D|E)+(-1)^{(|D|+1)(|E|+1)}(E|D)+D\{E\}$$
\begin{proposition}\label{prop:product on bar C}
The product $\bullet$ together with the comultiplication $\Delta$ makes $\Br(C^\bullet(A,A))$ an associative bialgebra.
\end{proposition}
Now let us explain how to modify the product $\bullet$ and to get a DG functor
\begin{equation}\label{eq:bul ca }\bullet: \Br({\bf C}^\bullet(A,B))\otimes \Br({\bf C}^\bullet(B,C))\to \Br({\bf C}^\bullet(A,C))
\end{equation}
\subsubsection{The \hfill brace \hfill operations \hfill on \hfill ${\bf C}^\bullet (A,B)$}\label{sss:brace for cats} \hfill For \hfill Hochschild \hfill cochains \hfill $D\in \newline C^\bullet (B, _{f_0} C_{f_1})$ and $E_i\in {\bf C}^\bullet (A, _{g_{i-1}} B _{g_i},$ $1\leq i\leq n,$ define the cochain
$$D\{E_1,\ldots, E_n\}\in C^\bullet (A,_{f_0g_0} C_{f_1g_n})$$
by
\begin{equation}\label{eq:braces for C(A,B)}
D\{E_1,\ldots, E_n\}(a_1,\ldots,a_N)=\sum\pm D(\ldots, E_1({\underline{\ldots}}), \ldots, E_n({\underline{\ldots}}), \ldots)
\end{equation}
where the space denoted by ${\underline{\ldots}}$ within $E_k({\underline{\ldots}})$ stands for $a_{i_k+1}, \ldots, a_{j_k},$ and the space denoted by $\ldots$ between $E_k({\underline { \ldots}})$ and $E_{k+1}({\underline{\ldots}})$ stands for $g_k(a_{j_k+1}, \ldots, ), g_k(\ldots), \newline \ldots, g_k(\ldots, a_{i_{k+1}}).$ The sum is taken over all possible combinations such that ${i_k}\leq {j_k}\leq i_{k+1}.$ The signs are as in \eqref{eq:Hochcochains15}.
\subsubsection{The $\bullet$ product on $\Br({\bf C}(A,B))$}\label{sss:bul for cats} For Hochschild cochains $D_i\in C^\bullet(B,_{f_{i-1}}C_{f_i})$ and $E_j\in C^\bullet(A,_{g_{j-1}}B_{g_j}),$ $1\leq i\leq m,$ $1\leq j\leq n,$ we have
$$(D_1|\ldots|D_m)\in \Br({\bf C}^\bullet (B,C))(f_0, f_m);$$
$$(D_1|\ldots|D_m)\in \Br({\bf C}^\bullet (A,B))(g_0, g_m);$$
define
$$(D_1|\ldots|D_m)\bullet(E_1|\ldots|E_n)\in \Br({\bf C}^\bullet (A,C))(f_0g_0, f_mg_n)$$
by the formula in the beginning of \ref{sss:bialg stree on C(A,A)}, with the following modification. The expression $D_i\{E_{j+1}, \ldots, E_k\}$ is now in ${\bf C}(A,C)(f_{i-1}g_{j+1}, f_ig_j),$ as explained above. The space denoted by $\ldots$ between $D_i\{E_{j+1}, \ldots, E_k\}$ and $D_{i+1}\{E_{p+1}, \ldots, E_q\}$ contains $f_i(E_{k+1}| \ldots)|f_i(\ldots)|\ldots| f_i(\ldots, E_p).$ Here, for an $A_\infty$ functor $f$ and for cochains $E_1, \ldots, E_k,$
\begin{equation}\label{eq:f of E}
f(E_1,\ldots, E_k)(a_1, \ldots, a_N)=\sum f(E_1(a_1, \ldots, a_{i_2-1}), \ldots, E_k(a_{i_{k}+1}, \ldots, a_n))
\end{equation}
The sum is taken over all possible combinations $1\leq i_1\leq i_2\leq\ldots i_k\leq n.$
\begin{lemma}\label{lemma:bullet for cats}
1) The product $\bullet$ is associative.

2) It is a morphism of DG cocategories. In other words,
one has
$$\Delta\circ \bullet=({\bullet_{13}\otimes\bullet_{24}})\circ (\Delta\otimes\Delta)$$
as morphisms
$$\Br(C^\bullet(A,B))(f_0,f_1)\otimes \Br(C^\bullet(B,C))(g_0,g_1)\to$$
$$\Br(C^\bullet(A,C))(f_0g_0,fg)\otimes \Br(C^\bullet(A,C))(fg,f_1g_1)$$
\end{lemma}
\subsubsection{Internal ${\underline{\rm{Hom}}}$ of DG cocategories}\label{sss:int hom}Following the exposition of \cite{Keller}, we explain the construction of Keller, Lyubashenko, Manzyuk, Kontsevich and Soibelman. 
For two $k$-modules $ V$ and $W$, let ${ {\Hom}}(V,W)$ be the set of homomorphisms from $V$ to $W$, and let $\iHom(V,W)$ be the same set viewed as a $k$-module. The two satisfy the property
\begin{equation}\label{eq:univ property of int hom}
\Hom(U\otimes V,W)\isomoto \Hom (U, \iHom(V,W)).
\end{equation}
In other words, $\iHom(V,W)$ is the internal object of morphisms in the symmetric monoidal category $k-\rm{mod}.$ The above equation automatically implies the existence of an associative morphism 
\begin{equation}\label{eq:composition on iHom}
\iHom(U,V)\otimes \iHom(V,W)\to \iHom(U,W)
\end{equation}
If we replace the category of modules by the category of algebras, there is not much chance of constructing anything like the internal object of morphisms. However, if we replace $k-\rm{mod}$ by the category of coalgebras, the prospects are much better. For our applications, it is better to consider counital coaugmented coalgebras. In this category, objects $\iHom$ do not exist because the equation \eqref{eq:univ property of int hom} does not agree with coaugmentations. However, as explained in \cite{Keller}, the following is true.
\begin{proposition}\label{prop:int Hom by Keller}
The category of coaugmented counital conilpotent cocategories admits internal $\iHom$s. For two DG categories $A$ and $B,$ one has
\begin{equation}\label{eq:iHom of cocats}
\iHom(\Br(A), \Br(B))=\Br({\bf C}(A,B))
\end{equation}
\end{proposition}
\subsection{Hochschild and cyclic complexes of DG categories and $A_\infty$ categories}\label{ss:Complexes of A infty cats}These are direct generalizations of the corresponding constructions for DG algebras. The Hochschild chain complex of a DG category (or, more generally, of an $A_\infty$ category) $A$ is defined as
$$C_p(A)=\bigoplus _{k-j=p} \bigoplus_{i_0,\ldots, i_p\in {\rm{Ob}}(A)} (A(i_0,i_1)\otimes {\overline A}(i_1,i_2)\otimes \ldots \otimes {\overline A}(i_p,i_0))^j;$$
the Hochschild cochain complex, as
$$C^p(A)=\prod _{k+j=p} \prod_{i_0,\ldots, i_p\in {\rm{Ob}}(A)} {\rm{Hom}}({\overline A}(i_0,i_1)\otimes \ldots \otimes {\overline A}(i_{p-1},i_p), A(i_0,i_p))^j;$$
the formulas for the differentials $b,$ $B,$ and $\delta$ are identical to those defined above for DG and $A_\infty$ algebras.
\section{Infinity algebras and categories}\label{s:Infinity-algebras and categories}
We develop a version of the definitions of an infinity algebra over an operad, an infinity category, and an infinity $n$-category. These
definitions
are closer to the work of Lurie, and of Batanin, than the ones developed in \ref{s:operads}. We compare the two. We show that Hochschild cochains of a DG algebra (or DG category) form an infinity two-category. We extend some of this discussion to the case of  Hochschild chains.
\subsection{Infinity algebras over an operad}\label{ss:Infty-algs} Let $\cP$ be an operad in sets. Define the category $\cP^{\#}$ as the PROP associated to $\cP.$ In other words, let $\cP^{\#}$ be the category whose objects are $[n], \, n=1,2,3,\ldots,$ and whose morphisms are defined by
\begin{equation}\label{eq:P sharp}
\cP^{\#}([n],[m])= \{{\rm{Natural}}\;{\rm{maps}}\;X^n\to X^m\}
\end{equation}
where $X$ is any set which is an algebra over $\cP.$ By this we mean that morphisms from $[n]$ to $[m]$ are all maps that you can construct universally, using the algebra structure, from $X^n$ to $X^m$ where $X$ is any set that is a $\cP$-algebra, so that every component $x_j$ in the argument $(x_1,\ldots,x_n)$ is used exactly once. 

\begin{remark}\label{rmk:As sharp}
When $\cP={\operatorname{As}}$, a $\cP$-algebra is an associative monoid. We will, however, modify the definition slightly and require it to be a unital monoid. The set of objects will be $\{[0], \,[1], \, [2], \ldots\}.$ Morphisms in ${\operatorname{As}}^{\#}([n],[m])$ can be identified with data
$$(f:\{1,\ldots, n\}\to \{1,\ldots, m\}; <_1, \ldots, <_m)$$
where $<_i$ is a linear order on $f^{-1}(\{i\}).$ A natural morphism associated to such data is defined by
\begin{equation}\label{eq:nat mm in As sharp}
(x_1,\ldots,x_n)\mapsto (\prod_{f(j)=1}x_j, \ldots, \prod_{f(j)=m}x_j)
\end{equation}
where the products are taken according to the orders $<_i$ and the product over the empty set is $1.$ This category was introduced in \cite{Fied}.
\end{remark}
The category $\cPsh$ has a symmetric monoidal structure as follows. On objects, $[n]\otimes [m]=[n+m];$ on morphisms, $f\otimes f':[n+n']\to[m+m']$ is the natural morphism obtained by concatenation of $f$ and $f'.$

The following definition is due to Leinster \cite{Leinster}.
\begin{definition}\label{dfn:inf alg}
Let ${\mathfrak C}$ be a symmetric monoidal category with weak equivalences.
An infinity-algebra over $\cP$ in ${\mathfrak C}$ is a functor
$$A:\cPsh\to {\mathfrak C}; \;[n]\mapsto A(n)$$
together with a natural transformation
\begin{equation}\label{eq:coprod Leinster}
\Delta(n,m):\,A(n+m)\to A(n)\otimes A(m)
\end{equation}
which is a weak equivalence for every pair $(n, \,m)$ and is coassociative, i.e.
$$ \id_{A(n)}\otimes \Delta(m,k)=\Delta(n,m)\otimes \id_{A(k)}:A(n+m+k)\to A(n)\otimes A(m) \otimes A(k)$$
\end{definition}
\begin{lemma} \label{lemma:infinity to infty}
For an infinity algebra $A$ in the category of complexes, there exists a $k[\cP]_\infty$-algebra structure on $A(1)$ such that the composition
$$\cP(n)\otimes A(n)\stackrel{\id _{\cP}\otimes\Delta}{\longrightarrow}\cP(n) \otimes A(1)^{\otimes n}\to A(1)$$
is homotopic to 
$$\cP(n)\otimes A(n)\to {\cP^{\#}}(n,1)\otimes A(n)\to A(1).$$
This structure can be chosen canonically up to homotopy.
\end{lemma}
\begin{proof} One can define the DG coalgebra 
$$\prod_n ({\rm{Bar}}({\mathcal P}) (n)\otimes A(n))^{\Sigma_n}$$
over ${\rm{Bar}}({\mathcal P})$ together with a coderivation $d$ of degree one and square zero, using the infinity algebra structure on $A.$ Then one transfers the DG coalgebra structure to the quasi-isomorphic complex 
$$\prod_n ({\rm{Bar}}({\mathcal P}) (n)\otimes A^{\otimes n})^{\Sigma_n}$$
which is the cofree coalgebra over ${\rm{Bar( {\mathcal P})}}$ generated by $A$. The resulting coderivation gives a ${\mathcal P}_\infty$-algebra structure on $A.$
\end{proof}
\begin{remark}\label{rmk:Costello vs Leinster}
In \cite{Costello TQFT}, Costello uses a different definition of an infinity algebra over a PROP in simplicial sets. For such a PROP ${\bf P},$ an infinity ${\bf P}$-algebra $A$ is defined as a functor ${\bf P}\to {\mathfrak C}$ together with an associative natural transformation $A(n)\otimes A(m)\to A(n+m)$ which is a weak equivalence for every $m$ and $n$. But, when ${\bf P}=\cP^{\#}$ for an operad $\cP,$ what we get is a strict algebra over $\cP.$
\end{remark}
\begin{remark}\label{rmk:Segal vs Leinster}\cite{Leinster}
When ${\mathfrak C}={\rm{Top}},$ then the definition of an infinity associative algebra leads to the definition of a Segal space $X$ with $X_0={\rm{pt}}.$ Indeed, put $X_n=A(n).$ Define $d_i:A(n)\to A(n-1)$ as follows. For $1\leq i\leq n-1,$ $d_i$ is induced by the map $(x_1,\ldots, x_n)\mapsto (x_1, \ldots, x_ix_{i+1}, \ldots, x_n)$ in ${\rm{As}}^{\#}([n],[n-1]).$ For $i=0,$ resp. $i=n,$ define $d_i$ to be the composition $A(n)\to A(1)\times A(n-1)\to A(n-1),$ resp. $A(n)\to A(n-1)\times A(1)\to A(n-1).$ Degeneracy operators $s_i$ are induced by maps $(x_1, \ldots, x_n)\mapsto (x_1, \ldots, x_i, 1,x_{i+1}, \ldots, x_n).$
\end{remark}
\subsubsection{Multiple infinity algebras}\label{sss:multiple algebras}
A morphism $A_1\to A_2$ of infinity $\cP$-algebras is a morphism of functors which is compatible with the underlying structure. By definition, a morphism is a weak equivalence if every map $A_1(n)\to A_2(n)$ is a weak equivalence.

Infinity $\cP$-algebras form a symmetric monoidal category: for two such algebras $A_1$ and $A_2, $ put $(A_1\otimes A_2)(n)=A_1(n)\otimes A_2(n);$ the action of morphisms from $\cP^{\#}$ and the comultiplication $\Delta$ are defined as tensor products of those for $A_1$ and $A_2.$

\begin{definition}\label{dfn:multiple algebras} An infinity $(\cP, \cQ)$-algebra is an infinity $\cP$-algebra in the symmetric monoidal category of infinity $\cQ$-algebras.
\end{definition}
In other words, an infinity $(\cP, \cQ)$-algebra is a collection of objects $A(m,n)$, morphisms $\cP^{\#}(m_1,m_2)\otimes A(m_1,n)\to A(m_2,n),$ and weak equivalences $\cQ^{\#}(n_1,n_2)\otimes A(m,n_1)\to A(m,n_2),$ $A(m_1+m_2,n)\to A(m_1,n)\otimes A(m_2,n),$ and $A(m,n_1+n_2)\to A(m,n_1)\otimes A(m,n_2)$ subject to various compatibilities.

\begin{example}\label{ex:mult algebra}
Let $\cP\otimes \cQ$ be the tensor product as in \cite{Dunn}; it is defined as the free product of $\cP$ and $\cQ$ factorized by relations
$$\alpha(\beta, \ldots, \beta)=\beta(\alpha, \ldots, \alpha)\in (\cP\otimes \cQ)(mn)$$
for all $\alpha\in \cP(m)$ and $\beta\in \cQ(n);$ here $\alpha(\beta, \ldots, \beta)$ denotes ${\rm{op}}_{n, \ldots, n}(\alpha\otimes (\beta\otimes \ldots \otimes \beta))$ and $\beta(\alpha, \ldots, \alpha)$ denotes ${\rm{op}}_{m, \ldots, m}(\beta\otimes(\alpha\otimes\ldots\otimes\alpha)).$ For a $\cP\otimes \cQ$-algebra $A$ one can define an infinity $(\cP, \cQ)$-algebra with $A(m,n)=A^{\otimes mn}.$
\end{example}

\subsection{Infinity categories}\label{ss:infty-cats}
\begin{definition}\label{dfn:As I} For a set $I$, let $\As^\#_I$ be the following category. Its objects are directed graphs with the set of vertices $I$ and with a finite number of edges. For two such graphs $\Gamma$ and $\Gamma ',$ $\As^\#_I (\Gamma, \Gamma ')$ is the set of all natural maps
$$\prod _{{\rm{edges} (\Gamma)}} X({\rm{source}}(e), {\rm{target}}(e))\to \prod _{{\rm{edges} (\Gamma ')}} X({\rm{source}}(e), {\rm{target}}(e))$$
for \hfill any \hfill category \hfill $X$ \hfill with \hfill ${\rm{Ob}}(X)=I;$ \hfill we \hfill require \hfill any \hfill argument $x_e\in X({\rm{source}}(e), \newline {\rm{target}}(e))$ to enter exactly once.
\end{definition}

Note that $\As^\#_I$ is a symmetric monoidal category if we put $\Gamma\otimes \Gamma '=\Gamma\coprod \Gamma '$ (disjoint union of edges with the same set of vertices). If $I$ is a one-element set then $\As^\#_I$ is the category ${\rm{As}}^\#$ as in \ref{s:Infinity-algebras and categories}.

A map of sets $F:I_1\to I_2$ induces a monoidal functor $F_*:\As^\#_{I_1}\to \As^\#_{I_2}.$

\begin{definition}\label{dfn:hom cat}
An infinity category $A$ in a symmetric monoidal category ${\mathfrak C}$ with weak equivalences is a set $I$ and a functor $A:\As^\#_I\to {\mathfrak C}$ together with a coassociative natural transformation
$$\Delta(\Gamma, \Gamma '): A(\Gamma\coprod \Gamma ')\to A(\Gamma)\otimes A( \Gamma ')$$
which is a weak equivalence for all $\Gamma, \Gamma '$ in ${\rm{Ob}}(\As^\#_I).$
\end{definition}
\subsection{Infinity $2$-categories}\label{ss:infty-2-cats} Let ${\mathfrak C}$ be the category of complexes, of simplicial sets, or of topological spaces. For an infinity category $A$ in ${\mathfrak C}$, define the homotopy category ${\rm{Ho}}(A)$ by
$${\rm{Ob}}{\rm{Ho}}(A)=I;\; {\rm{Ho}}(A)(i,j)=H^0(A(i\to j))$$
in the case of complexes, or $\pi_0$ in the other cases. (By $i\to j$ we denote the graph with two vertices marked by $i$ and $j$ and one arrow from $i$ to $j$).
\begin{definition}\label{dfn:mm and qeq of inf cats}
A morphism of infinity categories $(I_1,A_1)\to (I_2,A_2)$ is:

a) a map of sets $F:I_1\to I_2;$

b) a morphism of functors $A_1\to A_2\circ F_*$ which is compatible with $\Delta.$

A morphism is by definition a weak equivalence if it induces an equivalence of homotopy categories and every morphism $A_1(\Gamma)\to A_2(F_*(\Gamma))$ is a weak equivalence.
\end{definition}
The category of infinity categories is symmetric monoidal if one puts $(I_1, A_1) \otimes (I_2, A_2)=(I_1\times I_2, A)$ where
$$A(\Gamma)=A_1(\Gamma_1)\otimes A_2(\Gamma_2);$$
here $\Gamma_1$ has one edge $i_1\to j_1$ for every edge $(i_1,i_2)\to (j_1,j_2)$ and $\Gamma_2$ has one edge $i_2\to j_2$ for every edge $(i_1,i_2)\to (j_1,j_2).$
\begin{definition}\label{dfn:infty 2-cat} An infinity two-category is an infinity category in the symmetric monoidal category of infinity categories (the monoidal structure and weak equivalences on the latter are defined above).
\end{definition}
\subsection{Hochschild cochains as an infinity two-category} \label{ss:cochs as 2cat}
It is well known that categories \hfill form \hfill a \hfill two-category \hfill where \hfill one-morphisms \hfill are \hfill functors \hfill and \hfill two-\newline morphisms are natural transformations. Associative algebras also form a two-category: one-morphisms between $A$ and $B$ are $(A,B)$-bimodules; two-morphisms between $(A,B)$-bimodules $M$ and $N$ are morphisms of bimodules. In other words, to any algebras $A$ and $B$ we can associate a category $\cC(A,B)=(A,B)-{\operatorname{bimod}};$ for any three algebras there is a functor 
\begin{equation}\label{eq:product in 2-cat}
\cC(A,B)\times \cC(B,C)\to \cC(A,C)
 \end{equation}
 that satisfies the associativity property; it sends $(M,N)$ to $M\otimes _B N.$ For any $A$, there is the unit object $\id_A $ of $\cC(A)$ with respect to the above product. In fact $\id_A=A$ viewed as a bimodule. Note that
$$\End_{\cC(A,A)}(\id(A))={\operatorname{Center}}(A).$$
Note also that the two-category of algebras maps to the two-category of categories: an algebra $A$ maps to the category $A-{\rm{mod}},$ and a bimodule $M$ to the functor $M\otimes -.$

Our aim is to construct an infinity version of the above, namely an infinity 2-category whose objects are DG categories.

\subsubsection{The construction of the infinity 2-category of Hochschild cochains}\label{constr of 2-cat of Hochcoch}
Let $I$ be any set of DG categories. We first define the infinity-category ${\mathcal C}$ in the category of DG categories with the set of objects $I$. To do that, for any directed graph $\Gamma$ with set of vertices $I$ and with finitely many edges, put
\begin{equation}\label{eq:dfn of cat in algs}
{\mathcal C}(\Gamma)=\Cbr(\bigotimes _{{\rm{edges} (\Gamma)}} \Br({\bf C}({\rm{source}}(e), {\rm{target}}(e))))
\end{equation}
(recall that DG categories ${\bf{C}}(A,B)$ were defined in \ref{sss:bul for cats}).
For any $f:\Gamma\to \Gamma'$ in $\As^\#(\Gamma, \Gamma'),$ the corresponding map is induced by the $\bullet$ product and by insertion of $1$. The coproduct $\Delta:\cC(\Gamma\coprod \Gamma')\to \cC(\Gamma)\otimes \cC(\Gamma')$ is a partial case of the coproduct \eqref{eq:coprod coshuffle}.
\subsubsection{The module structure}\label{sss:The module structure} Similarly to the above, one can define the notion of an infinity algebra and an infinity module in a monoidal category $\cC$ with weak equivalences. Such an object is an infinity algebra $\{A(n)\}$ and a collection of objects $\{M(n-1,1)\}$ subject to various axioms that we leave to the reader. (Alternatively, one can replace the operad $\cP$ in Definition \ref{dfn:inf alg} by the colored operad $\As ^+$). Similarly one can define an infinity functor from an infinity category to $\cC.$ The latter is a collection of objects $M(\Gamma, v)$ where $\Gamma$ is a graph as above and $v$ is a vertex of $\Gamma.$ 

Recall that we have constructed in \ref{constr of 2-cat of Hochcoch} an infinity category ${\mathcal C}$ in the category ${\rm DGCat}$ of DG categories such that the its value at the graph $A\to B$ with two vertices and one edge is equal to
$${\mathcal C}(A\to B)=\Cbr\Br ({\bf C}(A,B)).$$ 
One can extend this definition by constructing an infinity functor $M$ from $\cC$ to ${\rm DGCat}$ such that $M(A)=\Cbr \Br (A).$ To do this, just observe that there is a morphism of DG cocategories
\begin{equation}\label{eq: bullet module}
\Br ({\bf C}(A,B))\times \Br(A)\to \Br(B)
\end{equation}
that agrees with the product from Lemma \ref{lemma:bullet for cats}. 
\subsubsection{The $A_\infty$ structure on chains of cochains}\label{sss:chains of cochains}
As a consequence of the above, we get
\begin{proposition}\label{prop: chains of cochains} 
1) The complex $C_{-\bullet}(C^\bullet (A,A), (C^\bullet (A,A))$ carries a natural $A_\infty$ algebra structure such that

\begin{itemize}
\item All $m_n$ are $k[[u]]$-linear, $(u)$-adically continuous
\item $m_1 = b+\delta +uB$
For $x,\;y \in C_{\bullet}(A)$: \item $(-1)^{|x|}m_2 (x,y) =
(\operatorname{sh} + u \operatorname{sh}')(x,y)$

For $D,\;E \in C^{\bullet}(A,A)$:
\item $(-1)^{|D|}m_2(D,E) = D \smile E$
\item $m_2(1\otimes D, \; 1 \otimes E) + (-1)^{|D||E|}m_2(1\otimes E, \; 1
\otimes D) = (-1)^{|D|}1 \otimes [D,\; E]$
\item $m_2( D, \; 1 \otimes E) + (-1)^{(|D|+1)|E|}m_2(1\otimes E, \;D) =
(-1)^{|D|+1}[D,\; E]$
\end{itemize}
(we use the shuffle products as defined in \ref{sss:prodcoprod}).

2) The complex $C_{-\bullet}(A,A)$ carries a natural structure of an $A_\infty$ module over the $A_\infty$ algebra from 1), such that

\begin{itemize}
\item All $\mu_n$ are $k[[u]]$-linear, $(u)$-adically continuous
\item $\mu_1 = b +uB$ on $C_{\bullet}(A)[[u]]$

For $a \in C_{\bullet}(A)[[u]]$:
\item $\mu_2 (a, D) = (-1)^{|a||D| + |a|}(i_D + u S_D)a$
\item $\mu_2 (a, 1 \otimes D) = (-1)^{|a||D|}L_D a $

For $a, \; x \in C_{\bullet}(A)[[u]]$: $(-1)^{|a|}\mu_2(a,x) =
(\operatorname{sh} + u \operatorname{sh}')(a,x)$
\end{itemize}

3) The above structures extend to negative cyclic complexes ${\rm CC}^-_\bullet.$
\end{proposition}
\begin{proof} In fact, the above is true if we replace $C_{-\bullet}$ or ${\rm CC}^-_\bullet$ by any functor which is multiplicative, i.e admits an associative K\"{u}nneth map.
\end{proof}
\begin{remark}\label{rmk:A infty on ch of coch}
An $\Ai$ structure as above was constructed in \cite{Ts GM}. It was used in \cite{DTT GM} to construct a Gauss-Manin connection on the periodic cyclic complex.
\end{remark}

\subsection{Hochschild chains}\label{ss:chains}
\subsubsection{A 2-category with a trace functor}\label{sss:A 2-category with a trace functor} The two-category of algebras and bimodules has an additional structure: a functor $\Tr _A: \cC(A,A)\to k-{\operatorname{mod}}$ such that the two functors
\begin{equation}\label{eq:trace functor 1}
\cC(A,B)\times \cC(B,A)\to \cC(A,A)\stackrel{\Tr_A}{\longrightarrow}k-{\operatorname{mod}}
\end{equation}
and
\begin{equation}\label{eq:trace functor 2}
\cC(A,B)\times \cC(B,A)\to \cC(B,A)\times \cC(A,B)\to \cC(B,B)\stackrel{\Tr_B}{\longrightarrow}k-{\operatorname{mod}}
\end{equation}
coincide. Here the first functor from the left in \eqref{eq:trace functor 1} and the second from the left in \eqref{eq:trace functor 2} are the products as in \eqref{eq:product in 2-cat}; the first functor from the left in \eqref{eq:trace functor 2} is the permutation of factors. We call a two-category with a functor as above {\em a two-category with a trace functor}.

For the two-category of algebras, the trace functor is defined as
\begin{equation}\label{eq:trace functor for bimods}
\Tr_A (M)=M/[A,M]=M\otimes _{A\otimes A^{\rm{op}}} A=H_0(A,M)
\end{equation}
\subsubsection{A dimodule over a 2-category}\label{A dimodule over a 2-category}
When we consider only those bimodules that come from morphisms of algebras, we get another algebraic structure on the two-category of algebras.

For an $(A,B)$-bimodule $M$, put
\begin{equation}\label{eq:dual bimodule}
M^{\vee}=\Hom(M,B)
\end{equation}
which is a $(B,A)$-bimodule. We have a morphism of $(B,B)$-bimodules
\begin{equation}\label{eq:pairing 1}
M^\vee \otimes _A M\to B
\end{equation}
For bimodules of the type that we will consider below, there is also a morphism of $(A,A)$-bimodules
\begin{equation}\label{eq:pairing 2}
A\to M \otimes _B M^\vee
\end{equation}
such that the compositions
\begin{equation}\label{eq:pairing 3}
M=A\otimes _A M \to (M\otimes _B M^\vee)\otimes _A M\isomoto M\otimes _B (M^\vee\otimes _A M)\to M
\end{equation}
\begin{equation}\label{eq:pairing 4}
M^\vee= M^\vee \otimes _A {\bf 1}\to M^\vee\otimes _A (M\otimes _B M^\vee)\isomoto (M^\vee\otimes _A M)\otimes _B M^\vee\to M^\vee
\end{equation}
are the identity morphisms. There is the second way to define a dual bimodule; namely, for an $(A,B)$-bimodule $M$, define a $(B,A)$-bimodule
\begin{equation}\label{eq: dual dagger}
M^\dagger =\Hom_A(M,A).
\end{equation}
There are bimodule morphisms $M\to M^{\vee\dagger}$ and $M\to M^{\dagger\vee}.$ The first one is an isomorphism for $M=_f B$ as above, the second for $M=B_f\isomoto (_f B)^\vee.$
Put
\begin{equation}\label{eq:pairing for bimodules}
\langle M,N\rangle = M \otimes _{B\otimes A^{\op}} N^\vee =(M \otimes _{B} N^\vee) \otimes _{A\otimes A^{\op}} A
\end{equation}
Let us describe the pairing $\langle M, N\rangle ,$ and the algebraic structure it is an example of, in the special case when our bimodules are of the form $_fB$ where $f$ is a homomorphism of algebras. Denote, as above, by $_f B _g$ the algebra $B$ viewed as an $A$-bimodule on which $A$ acts on the left via $f$ and on the right via $g$. Here $f$ and $g$ are two homomorphisms $A\to B.$ We have
\begin{equation}\label{eq:dfn of T(A,B) for algs1}
\langle{}_gB,{}_f B\rangle= \Tr_B (_f B _g)=B/\langle f(a)b-bg(a)|a\in A\rangle.
\end{equation}
Denote
\begin{equation}\label{eq:dfn of T(A,B) for algs2}
T(A,B)(f,g)=\langle{}_gB,{}_f B\rangle
\end{equation}
Note also that
\begin{equation}\label{eq:dfn of C(A,B) for algs, homoms}
\cC(A,B)(f,g)= \Hom_{A-B}(_f B, _gB)=\{b\in B|\forall a\in A: g(a)b=bf(a)\}.
\end{equation}
The collection $T(A, B)$ of $k$-modules $T(A,B)(f,g)$ carries the following structure.

1. For every $A$ and $B$, the collection ${T(A,B)}$ is a bimodule over the category $\cC(A,B).$

2. For every three algebras $A, \, B, \, C,$ there are pairings
\begin{equation}\label{eq:dimodule 1}
T(A,B)(g_0,g_1)\times \cC(B,C)(f_0,f_1)\to T(A,C)(f_0g_0, f_1g_1)
\end{equation}
and
\begin{equation}\label{eq:dimodule 2}
 T(A,C)(f_0g_0, f_1g_1)\times \cC(A,B)(g_1,g_0)\to T(B,C)(f_0,f_1)
 \end{equation}
 such that the following three compatibility conditions hold:

\begin{enumerate}
 \item
  the functors
 $$
 T(A,B)(h_0,h_1)\times \cC(B,C)(g_0,g_1)\times \cC(C,D)(f_0,f_1)\to
 $$
 $$
T(A,B)(h_0,h_1)\times \cC(B,D)(g_0h_0,g_1h_1)\to T(A,D)(f_0g_0h_0, f_1g_1h_1)
 $$
 and
 $$
 T(A,B)(h_0,h_1)\times \cC(B,C)(g_0,g_1)\times \cC(C,D)(f_0,f_1)\to
 $$
 $$
 T(A,C)(g_0h_0,g_1h_1)\times \cC(C,D)(h_0,h_1)\to T(A,D)(f_0g_0h_0, f_1g_1h_1)
 $$
 are equal;

 \item
  the functors
 $$T(A,D)(f_0g_0h_0, f_1g_1h_1)\times \cC(A,B)(h_1,h_0)\times \cC(B,C)(g_1,g_0)\to $$
 $$T(A,D)(f_0g_0h_0, f_1g_1h_1)\times \cC(A,C)(g_1h_0,g_1h_1)\to T(C,D)(f_0,f_1)$$
 and
 $$T(A,D)(f_0g_0h_0, f_1g_1h_1)\times \cC(A,B)(h_1,h_0)\times \cC(B,C)(g_1,g_0)\to $$
 $$T(B,D)(f_0g_0, f_1g_1)\times \cC(B,C)(g_1,g_0)\to T(C,D)(f_0,f_1)$$
 are equal;

 \item
 $$T(A,C)(g_0h_0, g_1h_1)\times \cC(A,B)(h_1,h_0)\times \cC(C,D)(f_0, f_1)\to$$
 $$T(B,C)(g_0,g_1)\times \cC(C,D)(f_0, f_1)\to T(B,D)(f_0g_0, f_1g_1)$$
 and
 $$T(A,C)(g_0h_0, g_1h_1)\times \cC(A,B)(h_1,h_0)\times \cC(C,D)(f_0, f_1)\to$$
 $$T(A,D)(f_0g_0h_0, f_1g_1h_1)\times \cC(A,B)(h_1,h_0)\to  T(B,D)(f_0g_0, f_1g_1)$$
 are equal.

\end{enumerate}
3. The pairings \eqref{eq:dimodule 1}, \eqref{eq:dimodule 2} are compatible with the $\cC(A,B)$-bimodule structures on $T(A,B).$

We call a 2-category and a collection of $T(A,B)(f,g)$ subject to the conditions above {\em a 2-category with a dimodule} (for want of a better term).

When $\cC$ is the 2-category of algebras and bimodules, and $T(A,B)(f,g)$ are as in \eqref{eq:dfn of T(A,B) for algs2}, then the action \eqref{eq:dimodule 1} is defined as
\begin{equation}\label{eq:action of 0-coch on 0-ch 1}
b\otimes c\mapsto f_1(b)c=cf_0(b)
\end{equation}
for $b\in {_{g_1}}B _{g_0}$ and $c\in \cC(B,C)(f_0,f_1);$
the action \eqref{eq:dimodule 2} is defined as
\begin{equation}\label{eq:action of 0-coch on 0-ch 2}
c\otimes b\mapsto f_1(b)c\eq cf_0(b)\in T(B,C)(g_1,g_0)
\end{equation}
for $b\in \cC(A,B)(g_1,g_0)=\{b\in B|\forall A:g_1(a)b=bg_0(a)$ and $c\in _{f_1g_1}C_{f_0g_0}.$

The definition of a dimodule is rather peculiar. If we replace categories $\cC(A,B)$ by sets, and therefore consider a category $\cC$ instead of a two-category, we get the definition of a $(\cC^{\op}, \cC)$-bimodule. In the case of 2-categories that we are working with, the notion of a dimodule is more subtle. If we put $$T^{\rm{dual}}(A,B)(f,g)=\Hom_k(T(A,B)(g,f),k)$$
then a dimodule defines two compatible actions
$$T(A,B)\times \cC(B,C)\to T(A,C)$$
$$\cC(A,B)\times T^{\rm{dual}}(B,C)\to T^{\rm{dual}}(A,C)$$

For any dimodule $T$ over a 2-category $\cC$, the action \eqref{eq:dimodule 2} of the morphism $\id_g,$ $f\in \Ob C(A,B),$ defines the morphism
the action \eqref{eq:dimodule 1} of the morphism $\id_f,$ $f\in \Ob C(B,C),$ defines the morphism
\begin{equation}\label{eq: dir im in dicat}
f_*: T(A,B)(g_0, g_1) \to T(A,C)(fg_0, fg_1);
\end{equation}
the action \eqref{eq:dimodule 2} of the morphism $\id_g,$ $g\in \Ob C(A,B),$ defines the morphism
\begin{equation}\label{eq: inv im in dicat}
g^*: T(A,C)(f_0g, f_1g)\to T(B,C)(f_0,f_1) .
\end{equation}
Our dimodule $T$ has the following extra property (which does not seem to follow from the axioms).
\begin{lemma}\label{lemma:cyclicity of the dimodule}
Let $f_0, f_1: B\to C$ and $g_0, g_1:A\to B$ be one-morphisms in $\cC$ such that $f_0g_0=f_1g_1.$ Then 
the diagram 
$$
\begin{CD}
T(A,B)(g_0,g_1) @>{f_1}_*>> T(A,C)(f_1g_0, f_1g_1)\\
@V{f_0}_*VV                   @V=VV       \\                                                    
T(A,C)(f_0g_0, f_0g_1) @>>>   T(A,C)(f_1g_0, f_0g_0)\\
@V=VV                      @V{g_0}^*VV\\
T(A,C)(f_1g_1, f_0g_1)@>{g_1}^*>>T(B,C)(f_1,f_0)
\end{CD}
$$
is commutative. 
\end{lemma}
\begin{proof} In fact, for $b\in _{g_1}B_{g_0},$ ${g_1}^*{f_0}_*(b)=f_0(b)\in{ _{f_0}}C_{f_1};$ ${g_0}^*{f_1}_*(b)=f_1(b)\in {_{f_0}}C_{f_1};$ the two are equal in $H_0(B, {_{f_0}C_{f_1}})$ (their difference is equal to the Hochschild chain differential of $1\otimes b;$ here is the origin of the cyclic differential $B,$ see below).
\end{proof}
\subsubsection{The higher structure on Hochschild chains: the first step}\label{sss:chains: the first step}
We expect that, when we replace $\cC(A,B)(f_0, f_1) $ by $C^\bullet(A,{_{f_1}B_{f_0}})$ and $T(A,B)(f_0, f_1) $ by $C_\bullet(A,{_{f_1}}B_{f_0}),$ the result will carry a structure of an infinity dimodule with property \eqref{lemma:cyclicity of the dimodule}. Observe first that the morphisms \eqref{eq: dir im in dicat}, \eqref{eq: inv im in dicat} can be written down easily:
\begin{equation}\label{eq: higher dir im in dicat}
f_*(b_0\otimes a_1 \ldots \otimes \otimes a_n)=f(b_0)\otimes a_1 \ldots \otimes a_n;
\end{equation}
\begin{equation}\label{eq: higher inv im in dicat}
g^*(c_0\otimes a_1 \otimes \ldots \otimes a_n)=c_0\otimes g(a_1) \otimes \ldots \otimes g(a_n).
\end{equation}
\subsubsection{The origin of the differential $B$}\label{sss:the origin of B}
Consider the statement of Lemma \ref{lemma:cyclicity of the dimodule} in the partial case $A=B=C,$ $f_1=g_0=f,$ $g_1=f_0=\id.$ We see that the two maps 
$$\id, \;f: C_\bullet (A, {_f}A)\to C_\bullet (A, {_f}A)$$
should be homotopic. Here
$$f(a_0\otimes a_1\otimes \ldots \otimes a_n)=f(a_0)\otimes f(a_1)\otimes \ldots \otimes f(a_n).$$
In particular, $C_\bullet (A,A)$ should carry an endomorphism of degree plus one. Such a homotopy can be easily written down as 
\begin{equation}\label{eq:B f}
B(f)(a_0\otimes a_1 \ldots \otimes a_n)=\sum_{i=0}^n (-1)^{ni}1\otimes f(a_i)\otimes \ldots \otimes f(a_n)\otimes a_0\ldots \otimes a_{i-1}
\end{equation}
\section{Deligne conjecture}\label{s:Deligne}
\subsection{Deligne conjecture for Hochschild cochains}\label{ss:Deligne cochains}
In the early 90s, Deligne conjectured that Hochschild cochains form a homotopy algebra over the operad of chain complexes of the little discs operad. This conjecture was proved by McClure and Smith in \cite{MS}. Subsequent proofs are contained in \cite{Batanin}, \cite{Berger}, \cite{HKV}, \cite{Kauf}, \cite{KoSo}, \cite{Lurie DAG6}, \cite{T}. 
\begin{thm}\label{thm:Deligne cochains} For any $A_\infty$ category $A$ there is an action of a cofibrant resolution of the DG operad $C_{-\bullet}(\LD_2)$ on the Hochschild complex $C^\bullet(A,A)$ such that at the level of cohomology:
\begin{enumerate}
\item
the generator of $H_0(\LD_2(2))$ acts by the cup product on $H^\bullet(A,A);$
\item
the generator of $H_1(\LD_2(2))$ acts by the Gerstenhaber bracket on $H^\bullet(A,A).$
\item
This structure is natural with respect to isomorphisms.
\end{enumerate}
\end{thm}
\subsection{Deligne conjecture for Hochschild chains}\label{ss:Deligne chains}
An extension of the Deligne conjecture to chains maintains that the pair of complexes of Hochschild cochains and chains is a homotopy algebra over the two-colored operad of little discs and cylinders. 
\begin{thm}\label{thm:Deligne chains} For any $A_\infty$ category $A$ there is an action of a cofibrant resolution of the DG operad $C_{-\bullet}(\LD_2)$ on the pair of Hochschild complexes $(C^\bullet(A,A), C_{-\bullet}$ $(A,A)$ such that at the level of cohomology:
\begin{enumerate}
\item
the generator of $H_0(\LC(1,1))$ acts by the pairing $H^\bullet(A,A)\otimes H_{-\bullet}(A,A)\to H_{-\bullet}(A,A);$
\item
the generator of $H_1(\LC(1,1))$ acts by the pairing $H^\bullet(A,A)\otimes H_{-\bullet}(A,A)\to H_{-\bullet+1}(A,A).$
\item
This structure is natural with respect to isomorphisms.
\end{enumerate}
\end{thm}
\section{Formality of the operad of little two-discs}\label{s:Formality of the little discs operad}
\subsection{Associators}\label{ss: Associators} We follow the exposition in \cite{BN}, \cite{T1}, and \cite{SW}.
\subsubsection{The operad in categories $\PaB$}\label{ss:PaB}  Define the category $\PaB(n)$ as follows. Its object is a parenthesized permutation, i.e. a pair $(\sigma, \pi)$ of a permutation $\sigma\in S_n$ and a parenthesization $\pi$ of length $n.$ A parenthesization is by definition an element of the free non-associative monoid with one generator $\bullet$. Example ($n=6$):
\begin{equation}\label{eq:parenthes}
\pi=((\bullet\bullet)((\bullet\bullet)(\bullet\bullet)))
\end{equation}
A morphism from $(\sigma_1, \pi_1)$ to $(\sigma_2, \pi_2)$ is an element of the braid group $B_n$ whose projection to $S_n$ is equal to $\sigma_2 ^{-1}\sigma_1.$ The composition of morphisms is given by the multiplication of braids. 

To describe the operadic structure, it is more convenient to use a slightly different definition of $\PaB(n).$ 
A parenthesization of a finite ordered set $A$ is a parenthesization of length $n=|A|$ where the $jth$ symbol $\bullet$ is replaced by $a_j$ for all $j.$

For two total orders $<_1$ and $<_2$ on a finite set $A$, a pure braid between $(A, <_1)$ and $(A, <_2)$ is a braid whose lower ends are decorated by elements of $A$ in the order $<_1,$ whose  upper ends are decorated by elements of $A$ in the order $<_2,$ and whose strands all go from $a$ to the same element $a$. For a finite set $A$, the category $\PaB(A)$ is defined as follows:
\begin{enumerate} 
\item
Objects of $\PaB(A)$ are pairs $(<, \pi)$ where $<$ is a total order on $A$ and a parenthesization of $A$;
\item
a morphism from $(<_1, \pi_1)$ to $(<_2, \pi_2)$ is a pure braid from $(A, <_1)$ to $(A, <_2);$
\item
the composition is the multiplication of braids.
\end{enumerate}

Now let us define the operadic composition. Let $A$ and $B$ be totally ordered finite sets. Consider the surjection $A\coprod B\to A\coprod \{c\}$ that is the identity on $A$ and that sends all elements of $B$ to $c.$ The operadic composition
\begin{equation}\label{eq:circ j for PaB}
\PaB(B)\times\PaB(A\coprod \{c\})\to \PaB(A\coprod B)
\end{equation}
 corresponding to this surjection acts as follows: Let $<_1$ be a total order on $B$, $\pi_1$ a parenthesization of $B,$  $<_2$ a total order on $A\coprod \{c\},$ and $\pi_2$ a parenthesization of $A\coprod \{c\}$. Then the value of the functor \eqref{eq:circ j for PaB} on $((<_1,\pi_1), (<_2,\pi_2)$ is $(<, \pi)$ where
 \begin{enumerate}
 \item $<$ is the total order for which $a<a'$ iff $a<_2 a';$ $b<b'$ iff $b<_1 b';$ $a<b$ iff $a<_2 c;$
 \item $\pi$ is obtained from the parenthesization $\Pi_2$ by replacing the symbol $c$ with the set $B$, parenthesized by $\pi_1.$
 \end{enumerate}
 Note that the operad of sets $\Ob\PaB$ is the free operad generated by one binary operation. At the level of morphisms, let $\gamma$ be a pure braid between $(B, <_1)$ and $(B, <_1')$; let $\gamma '$ be a pure braid between $(A\coprod \{c\}, <_2)$ and $(A\coprod \{c\}, <_2').$ The functor \eqref{eq:circ j for PaB} sends $(\gamma, \gamma')$ to $\gamma''$ defined as $\gamma'$ in which the strand from $c$ to $c$ is replaced by the pure braid $\gamma.$
\subsubsection{The operad in Lie algebras $\frt$}\label{ss:tt} For a finite set $A,$ let $\frt(A)$ be the Lie algebra with generators $t_{ij},\,i,\,j\in A,$ subject to relations 
\begin{equation}\label{eq:relations in t(n) 1} 
[t_{ij}, t_{kl}]=0
\end{equation}
if $i,j,k,l$ are all different;
\begin{equation}\label{eq:relations in t(n) 2} 
[t_{ij}, t_{ik}+t_{jk}]=0
\end{equation}
if $i,j,k$ are all different. We put $\frt(n)=\frt(\{1,\ldots, n\}).$ These Lie algebras form an operad in the category of Lie algebras where the monoidal structure is the direct sum. The operadic compositions are uniquely defined by the compositions
$\circ_j : \frt(m)\oplus \frt(n)\to \frt(n+m-1)$
acting as follows. Let $A$ and $B$ be finite sets. Consider the surjection $A\coprod B\to A\coprod \{c\}$ that is the identity on $A$ and that sends all elements of $B$ to $c.$ The operadic composition
$\frt(B)\oplus\frt(A\coprod \{c\})\to \frt(A\coprod B)$
 corresponding to this surjection acts as follows:
\begin{equation}\label{eq:operadic comp in tt}
(t_{bb'}, t_{aa'})\mapsto t_{bb'}+ t_{aa'}; (t_{bb'}, t_{ac})\mapsto t_{bb'}+\sum_{b''\in B}t_{ab''}
\end{equation}
for $a,\,a'\in A, b,\,b'\in B.$ The action of the symmetric group on $U(\frt(n))$ is by permutation of pairs of indices $(ij).$

The operad $\frt$ gives rise to the operads $U(\frt)$ and ${\widehat{U(\frt)}}$ in the category of algebras and to the operad ${\widehat{U(\frt)}}^{\operatorname{group}}$ in the category of groups. Here $U(\frt)$ is the universal enveloping algebra of $\frt,$ ${\widehat{U(\frt)}}$ its completion with respect to the augmentation ideal, and ${\widehat{U(\frt)}}^{\operatorname{group}}$ the set of grouplike elements of this completion (with respect to the coproduct for which all $t_{ij}$ are primitive). Since every group is a category with one object, we can consider ${\widehat{U(\frt)}}^{\operatorname{group}}$ as an operad in categories.
\subsubsection{Definition of an associator}\label{sss:dfn ass} Let $\sigma$ be the morphism in $\PaB(2)$ between $(12)$ and $(21)$ corresponding to the generator of the pure braidgroup ${\operatorname{PB}}_2\isomoto \bbZ.$ Let $a$ be the morphism in $\PaB(3)$ between $(12)3)$ and $(1(23))$ corresponding to the trivial pure braid $e.$ 
\begin{definition}\label{dfn:associator}
An associator is a group element $\Phi\in {\widehat{U(\frt(3))}}^{\operatorname{group}}$ such that there is a morphism of operads in categories
$$\PaB\to {\widehat{U(\frt)}}^{\operatorname{group}}$$
that sends $\sigma$ to $\exp(\frac{t_{12}}{2})$ and $a$ to $\Phi.$
\end{definition}
The following theorem is essentially proven in \cite{Dr}. It is formulated in the language of operads in \cite{T1} which is based on \cite{BN}.
\begin{thm}\label{thm:associator exists} 
There exists an associator $\Phi.$
\end{thm}
\begin{remark}\label{rmk:Lie theory and assoc}
The above theorem is plausible because the relations \eqref{eq:relations in t(n) 1}, \eqref{eq:relations in t(n) 2} are infinitesimal analogs of the defining relations in pure braid groups. Na\"{i}vely, $\frt(n)$ is the Lie algebra of ${\operatorname{PB}}_n.$ If the latter were nilpotent, the theorem would follow from rational homotopy theory. However, pure braid groups are far from being nilpotent, so the existence of an associator is not easy to prove.
\end{remark}
\subsubsection{Parenthesized braids and little discs}\label{sss:PaB vs FM} Consider the embedding 
\begin{equation}\label{eq:FM1 FM2}
\FM_1\to \FM_2
\end{equation}
induced by the embedding $\bbR\to \C\isomoto \bbR^2.$ Note that the zero strata of $\FM_1$ form an operad in sets that is isomorphic to the operad $\Ob\PaB.$ We denote this suboperad of $\FM_1$ (and $\FM_2$) by $\PaP.$ Denote by $\pi_1(\FM_2(n), \PaP(n))$ the full subcategory of the fundamental groupoid of $\FM_2$ with the set of objects $\PaP(n).$ The collection of categories $\pi_1(\FM_2(n), \PaP(n))$ is an operad that we denote by $\pi_1(\FM_2, \PaP)$.
\begin{lemma}\label{lemma:PaP}
There is an isomorphism of operads in categories
$$\pi_1(\FM_2, \PaP)\isomoto \PaB$$
\end{lemma}
\subsection{Formality of the operad of chains of little two-discs}\label{ss:Formality of little discs operad} 
\begin{thm}\label{thm:formality of LD2} \cite{T1}
There is a chain of weak equivalences between DG operads $C_{-\bullet}(\LD_2)$ and $H_{-\bullet}(\LD_2).$
\end{thm}
\begin{proof} There is a chain of equivalences of topological operads:
$${\operatorname{Nerve}}\;\pi_1(\FM_2, \PaP)\lisomoto{\operatorname{Nerve}}\;\pi_1(\FM_2)\lisomoto \FM_2$$
The morphism of nerves on the left is induced by an equivalence of categories and therefore an equivalence. The map on the right is the classifying map which is an equivalence because all $\FM_2(n)$ are $K(\pi, 1).$ By Lemma \ref{lemma:PaP}, there is a chain of equivalences between $\Nerve \;\PaB$ and $FM_2.$ Applying the functor $C_{-\bullet} $ to this chain of equivalences, we see that  it is enough to construct a chain of weak equivalences between $C_{-\bullet} (\Nerve \;\PaB)$ and $H_{-\bullet}(\LD_2),$ which is the same as $H_{-\bullet}(\FM_2).$ An associator $\Phi$ provides an equivalence
\begin{equation}\label{eq:Assoc as equiv of ops}
C_{-\bullet} (\Nerve \;\PaB)\isomoto C_{-\bullet} (\Nerve \;{\widehat{U(\frt)}}^{\operatorname{group}})
\end{equation}
The right hand side of the above (if we replace singular chains of the geometric realization by simplicial chains) is the completed version of the chain complex of the group ${\widehat{U(\frt)}}^{\operatorname{group}}.$ It is not difficult to define the chain of equivalence below, where $\Cbr$ stands for the cobar construction of the augmentation ideal or, what is the same, the standard complex for computing $\Tor_{-\bullet}^{U(\frt)}(k,k).$
$$C_{-\bullet} (\Nerve \;{\widehat{U(\frt)}}^{\operatorname{group}})\isomoto {\widehat{\Cbr}}_{-\bullet}(U({\frt})^+)\lisomoto {\Cbr}_{-\bullet}(U({\frt})^+)\lisomoto C^{\Lie}_{-\bullet}(\frt)$$
Finally, the right hand side is quasi-isomorphic to $\Gerst\isomoto H_{-\bullet}(\LD_2).$
\end{proof}
\subsection{Formality of the colored operad of little discs and cylinders}\label{ss:Formality of the colored operad of little discs and cylinders}
\begin{thm}\label{thm:formality of LC}
There are chains of weak equivalences between two-colored DG operads $C_{-\bullet}(\LC)$ and $H_{-\bullet}(\LC),$ and between $C_{-\bullet}(\LfC)$ and $H_{-\bullet}(\LfC).$
\end{thm}
\begin{proof} The proof for the case of $\LC$ is virtually identical to the proof of Theorem \ref{thm:formality of LD2}. The proof for $\LfC$ requires a modification regarding the action of $S^1.$ We omit it here.
\end{proof}
\subsubsection{Gamma function of an associator}\label{sss:Gamma fn of associator} Note that 
$$\frt(3)\isomoto {\operatorname{Free}}{\operatorname{Lie}}(t_{12}, t_{23})\oplus k\cdot (t_{12}+t_{13}+t_{23})$$
It is easy to see \cite{Dr}, \cite{BN} that one can choose $\Phi=\Phi(t_{12}, t_{23}).$ Since $\Phi$ is grouplike, $\log \Phi$ is a Lie series in two variables. Put
\begin{equation}\label{eq:zetas}
\log\Phi(x,y)=-\sum_{k=1}^\infty \zeta_\Phi (k+1)\ad^k_x(y)+O(y^2)
\end{equation}
and 
\begin{equation}\label{eq:zetas 1}
\Gamma_\Phi (u)=\exp(\sum_{n=2}^\infty (-1)^n \zeta_\Phi(n) u^n/n)
\end{equation}
It is known that
\begin{equation}\label{eq:zetas 2}
\exp(\sum_{n=1}^\infty  \zeta_\Phi(2n) u^{2n})=-\frac{1}{2}(\frac{u}{e^u-1}-1+\frac{u}{2})
\end{equation}

\section{Noncommutative differential calculus} \label{ncdc}

We deduce from \ref{s:Deligne} and \ref{s:Formality of the little discs operad} that the Hochschild cochain complex is an infinity Gerstenhaber algebra and, more generally, the pair of the cochain and the chain complexes is an infinity calculus. This admits the interpretation below, due to the fact that infinity algebras can be rectified (cf. \ref{s:operads}).

\subsection{The $\Gerst_\infty$ structure on Hochschild cochains} \label{ss:gerstenhaber-1}
Below is the theorem from \cite{T}.
\begin{thm}   \label{thm:gerstenhaber-1}
For every associative algebra $A$ and every associator $\Phi,$ there exists a $\Gerst_\infty$ algebra structure on $C^\bullet(A,A)$, natural with respect to isomorphisms of algebras, such that
\begin{enumerate}
\item The induced Gerstenhaber
algebra structure on 
$H^{\bullet}(A,A)$  
is the standard one, defined by the cup product and the Gerstenhaber bracket as in \ref{hocochain}.
\item
The underlying $\Li$ structure on $C^{\bullet+1}(A,A)$ is given by the Gerstenhaber bracket.
\end{enumerate}
\end{thm}


\subsection{The $\Cai $ structure on Hochschild chains}  \label{ss:calculus-2}

\begin{thm} \label{thm:calc} \cite{TT}, \cite{DTT Fty}
For every associative algebra $A$ and every associator $\Phi,$ there exists a $\Cai$ algebra structure on $(C^\bullet(A,A),C_\bullet (A,A))$, such that
\begin{enumerate}
\item
The induced calculus structure on $(H^\bullet(A,A),H_\bullet (A,A))$ is defined by the Gerstenhaber bracket, the cup product, the actions $\iota_D$ and $L_D$ from \ref{pairings}, and the cyclic differential $B$, as in Example \ref{ex:calc-0}.
\item
The induced structure of an $\Li$ module over $C^{\bullet+1}(A,A)$ on $C_{{\bullet}}(A)[[u]]$ is defined by the differential
$b+uB$ and the DG Lie algebra action $L_D$ from \ref{pairings}.
\end{enumerate}
\end{thm}

\subsection{Enveloping algebra of a Gerstenhaber algebra}
\label{ss:enveloping}
The following construction is motivated by Example
\ref{ex:calc-M}. For a Gerstenhaber algebra ${\mathcal{V}}^{\bullet}$,
let $\Y ({\mathcal{V}}^{\bullet})$ be the associative algebra
generated by two sets of generators $i_a$, $L_a$, $a \in
{\mathcal{V}}^{\bullet}$, both $i$ and $L$ linear in $a$,
$$|i_a| = |a|; \; |L_a| = |a| - 1$$
subject to relations
$$i_ai_b = i_{ab};\;\;[L_a,L_b] = L_{[a,b]};$$
$$[L_a, i_b] = i_{[a,b]};\;L_{ab} = (-1)^{|b|}L_ai_b + i_aL_b$$

The algebra $\Y ({\mathcal{V}}^{\bullet})$ is equipped with the differential $d$
of degree one which is defined as a derivation sending $i_a$ to $(-1)^{|a|-1}L_a$ and
$L_a$ to zero.

For a smooth manifold $M$ one has a homomorphism
$$\Y (\VM) \rightarrow \D (\Omega ^{\bullet}(M))$$
The right hand side is the algebra of differential operators on
differential forms on $M$, and the above homomorphism sends the
generators $i_a$, $L_a$ to corresponding differential operators on
forms (cf. Example  \ref{ex:calc-M}). It is easy to see that the
above map is in fact an isomorphism.
\subsubsection{Differential operators on forms in noncommutative calculus}\label{sss:ops on forms in nc calc}
Using a standard rectification argument 
one can restate Theorem \ref{thm:calc} as follows:
\begin{thm} \label{thm:calc rectified} 
For every associative algebra $A$ and every associator $\Phi,$ there exists a DG calculus $(\cV^\bullet (A), \Omega^\bullet (A)),$ natural with respect to isomorphisms of algebras, such that:

1) there is a quasi-isomorphism of DGLA 
 $$\cV^{\bullet+1} (A)\to C^{\bullet+1}(A,A)$$
 and a compatible quasi-isomorphism of DG modules
 $$(\Omega^\bullet (A)[[u]], \delta+ud)\to (C_\bullet(A,A)[[u]], b+uB)$$
 where the right hand sides are equipped with the standard structures given by the Gerstenhaber bracket and the operation $L_D;$ both maps are natural with respect to isomorphisms of algebras;
 
 2) The statement 1) of Theorem \ref{thm:calc} holds.
 \end{thm}
 \begin{proposition}\label{prop:nc ops on forms}
 There is an $\Ai$ quasi-isomorphism of $\Ai$ algebras, natural with respect to isomorphisms of algebras:
 $$Y(\cV^\bullet(A))\to C_{-\bullet}(C^\bullet (A,A), C^\bullet (A,A))$$
 that extends to an $\Ai$ quasi-isomorphism
 $$(Y(\cV^\bullet(A))[[u]], \delta+ud)\to {\rm {CC}}^-_{-\bullet}(C^\bullet (A,A), C^\bullet (A,A))$$
 (the $\Ai$ structures on the right hand side were defined in \ref{prop: chains of cochains} ).
 \end{proposition}
 The proof is given in \cite{TT4}.
\section{Formality theorems}\label{ss:formality}
For an associative algebra $A$ and an associator $\Phi,$ let 
$$(C^\bullet(A,A),C_\bullet(A,A))_\Phi$$ 
denote the $\Cai$ algebra given by Theorem \ref{thm:calc}. Let $X$ be a smooth manifold (real, complex analytic, or algebraic over a field of characteristic zero).
\begin{thm}\label{thm:formality}
There is a $\Cai$ quasi-isomorphism between the sheaves of $\Cai$ algebras $(C^\bullet(\cO_X, \cO_X),C_\bullet(\cO_X, \cO_X))_\Phi$ and $\Ca _X$ such that:
\begin{enumerate}
\item
the induced isomorphism
$${\bf H}^\bullet(\cO_X, \cO_X)\to H^\bullet(X, \wedge ^\bullet T_X)$$
is given by 
$$c\mapsto \iota({\sqrt{\widehat A}}_\Phi(T_X))  I_{\rm{HKR}}(c);$$
\item
the induced isomorphism
$${\bf H}_\bullet(\cO_X, \cO_X)\to H^{-\bullet}(X, \Omega)$$
is given by 
$$c\mapsto{\sqrt{\widehat A}}_\Phi(T_X)\wedge I_{\rm{HKR}}(c)$$
\end{enumerate}
where the left hand side stands for the hypercohomology of $X$ in the sheaf of Hochschild complexes, and ${\sqrt{\widehat A}}_\Phi(T_X)$ is the characteristic class of the tangent bundle $T_X$ corresponding to the symmetric power series $\Gamma_\Phi (x_1)\ldots \Gamma_\Phi(x_n)$. Here $\Gamma_\Phi$ denotes the gamma function of the associator $\Phi.$
\end{thm}
The proof can be obtained from \cite{Wil}, \cite{Wil1}.

\section{Deformation quantization} \label{ss:defqua}
Let $M$ be a smooth manifold. By a deformation quantization of $M$ we mean a
formal product
\begin{equation} \label{eq:starproduct}
f * g = fg + \sum_{k=1}^{\infty} (i\hbar)^k P_k(f,g)
\end{equation}
where $P_k$ are bidifferential expressions, $*$ is associative,
and $1*f = f*1 = f$ for all $f$. Given such a product (which is
called a star product), we define
\begin{equation} \label{eq:At(M)}
{\mathbb{A}}^\hbar(M) = (C^{\infty}(M)[[\hbar]], *)
\end{equation}
This is an associative algebra over ${\mathbb{C}}[[\hbar]]$. By
${\mathbb{A}}_c^\hbar(M)$ we denote the ideal $C^{\infty}_c(M)[[\hbar]]$ of this
algebra. An isomorphism of two deformations is by definition a power series
$T(f) = f + (i\hbar)^k\sum_{k=1}^{\infty} T_k(f)$ where all $T_k$ are differential
operators and which is an isomorphism of algebras.

Given a star product on $M$, for $f,\;g \in C^{\infty}(M)$ let
\begin{equation} \label{equ:poissonbracket}
\{f,g\}= P_1(f,g) - P_1(g,f) = \frac{1}{t}[f,g] |\hbar=0.
\end{equation}
This is a Poisson bracket corresponding to some Poisson structure
on $M$. If this Poisson structure is defined by a symplectic form
$\omega$, we say that ${\mathbb{A}}^\hbar(M)$ is a deformation of the
symplectic manifold $(M,\;\omega)$.

Recall the following classification result from \cite{Del},
\cite{DWL}, \cite{Fe}, \cite{NT4}.
\begin{thm} \label{thm:classification of symplectic deformations}
Isomorphism classes of deformation quantizations of a symplectic
manifold $(M, \omega)$ are in a one-to-one correspondence with the
set
$$ \frac{1}{i\hbar}[\omega] + H^2 (M, {\mathbb{C}}[[\hbar]])$$
where $[\omega]$ is the cohomology class of the symplectic
structure $\omega$.
\end{thm}

In defining the Hochschild and cyclic complexes, we use
$k={\mathbb{C}}[[\hbar]]$ as the ring of scalars, and put
\begin{equation} \label{eq:tensdefholo}
{\mathbb{A}}^\hbar(M) ^{\otimes n} =
\operatorname{jets}_{\Delta}C^{\infty}(M^n)[[\hbar]]
\end{equation}
Sometimes we are interested in the homology defined using $\C$ as the ring of scalars. Then we use the standard definitions where the tensor products over $\C$ are defined by 
\begin{equation} \label{eq:tensdefholo C}
{\mathbb{A}}^\hbar (M) ^{\otimes_{\C} n} =
\operatorname{jets}_{\Delta}C^{\infty}(M^n)[[\hbar_1, \ldots, \hbar_n]]
\end{equation}
Let ${\mathbb{A}}^\hbar(M)$ be a deformation of a symplectic manifold
$(M, \omega)$.
\begin{thm} \label{thm:deformations}
There exists a quasi-isomorphism
$$C_{\bullet}({\mathbb{A}}^\hbar(M), {\mathbb{A}}^\hbar(M))[\hbar^{-1}] \rightarrow
(\Omega^{2n-\bullet}(M)((\hbar)), i\hbar d)$$
which extends to a ${\mathbb{C}}[[\hbar,u]]$-linear, $(\hbar,u)$-adically continuous
quasi-isomorphism
$${\operatorname{CC}}^-_{\bullet}({\mathbb{A}}^\hbar(M))[\hbar^{-1}] \rightarrow (\Omega^{2n-\bullet}(M)[[u]]((\hbar)), i\hbar d)$$
\end{thm}

An analogous theorem holds for ${\mathbb{A}}_c^\hbar(M)$ if we replace $\Omega
^{\bullet}$ by $\Omega^{\bullet}_c$.
\subsubsection{The canonical trace} \label{sss:trace}
Combining the first map from Theorem \ref{thm:deformations} in the compactly
supported case with integrating over $M$ and dividing by $\frac{1}{n!}$, one
gets the canonical trace of Fedosov
$$ Tr : {\mathbb{A}}_c^\hbar(M) \rightarrow {\mathbb{C}}((\hbar)) $$
It follows from Theorem \ref{thm:deformations} that, for $M$
connected, this trace is unique up to multiplication by an element
of ${\mathbb{C}}((i\hbar))$.

\section{Applications of formality theorems to deformation quantization}\label{s:Applications to deformation quantization}
\subsection{Kontsevich formality theorem and classification of deformation quantizations}\label{ss:classification of deformation quantizations} From Theorem \ref{thm:formality} we recover the formality theorem of Kontsevich \cite{Konts1}, \cite{Konts}:
\begin{thm}\label{thm:foKo}
For a $C^\infty$ manifold $X$ there exists an $L_\infty$ quasi-isomorphism of DGLA
$$\Gamma(X, \wedge^{\bullet+1}(T_X))\to C^{\bullet+1}(C^\infty(X), C^\infty(X))
$$
For a complex manifold $X$, or for a smooth algebraic variety $X$ over a field of characteristic zero, there exists an $L_\infty$ quasi-isomorphism of sheaves of DGLA
$$\wedge^{\bullet+1}(T_X)\to C^{\bullet+1}(\cO_X, \cO_X)$$
\end{thm}
\begin{definition}\label{dfn:formal Poiss} A formal Poisson structure on a $C^\infty$ manifold $X$ is a power series $\pi=\sum _{n=0}^{\infty} (i\hbar)^{n+1}\pi_n$ where $\pi_n$ are bivector fields and $[\pi,\pi]_{\rm{Sch}}=0$ (here $[\;]_{\rm{Sch}}$ denotes the Schouten bracket, extended  bilinearly to power series in $\hbar$ with values in multivector fields). An equivalence between two formal Poisson structures $\pi$ and $\pi '$ is a series $X=\sum_{n=1}^{\infty} (i\hbar)^{n+1}X_n$ such that $\pi '=\exp(L_X)\pi.$
\end{definition}
From Theorem \ref{thm:foKo} one deduces \cite{Konts1}, \cite{Konts}
\begin{thm}\label{thm:cla def qua}
There is a bijection between isomorphism classes of deformation quantizations of a $C^\infty$ manifold $X$ and equivalence classes of formal Poisson structures on $X.$
\end{thm}
This theorem admits an analog for complex analytic manifolds and for smooth algebraic varieties in characteristic zero. The correct generalization of a deformation quantization is a formal deformation of the structure sheaf $\cO_X$ as an {\em algebroid stack} (cf. \cite{Kashi}, \cite{K2} for definitions).
\begin{thm}\label{thm:cla def qua holo} \cite{BGNT}, \cite{BGNT1}
For any associator $\Phi,$ there is a bijection between isomorphism classes of deformation quantizations of a complex manifold $X$ and equivalence classes of Maurer-Cartan elements of the DGLA 
$$(\hbar\Omega^{0,\bullet}(X, \wedge^{\bullet +1}T_X)[[\hbar]], {\overline{\partial}})$$
\end{thm}

\subsubsection{Hochschild cohomology of deformed algebras}\label{sss:Hochschild cohomology of deformed algebras} Let $\pi$ be a formal Poisson structure on a smooth manifold $X.$ Denote by $\bbA^\pi$ the deformation quantization algebra given by Theorem \ref{thm:cla def qua}. The Hochschild cochain complex $C^{\bullet}(\bbA^\pi,\bbA^\pi)$ is by definition the complex of multidifferential, $\C[[\hbar]]$-linear cochains. One deduces from Theorem \ref{thm:foKo}
\begin{thm}\label{thm:hoco of def alg} \cite{Konts1}, \cite{Konts}
There is an $L_\infty$ quasi-isomorphism of DGLA
$$(\Gamma(X, \wedge^{\bullet+1}(T_X))[[\hbar]], [\pi,-]_{\rm{Sch}})\isomoto C^{\bullet}(\bbA^\pi,\bbA^\pi) $$
\end{thm}
\subsection{Formality theorem for chains and the Hochschild and cyclic homology of deformed algebras}\label{ss:Hochschild and cyclic homology of deformed algebras}
Note that, by Theorems  \ref{thm:foKo} and \ref{thm:hoco of def alg}, the Hochschild and negative cyclic complexes of $C^\infty(X),$ resp. of $\bbA^\pi,$ are $L_\infty$-modules over $\Gamma^{\bullet+1}(T_X),$ resp. over  $(\Gamma(X, \wedge^{\bullet+1}(T_X))[[\hbar]], [\pi,-]_{\rm{Sch}}).$

\begin{thm}\label{thm:hoho of def alg} \cite{D2}, \cite{D}, \cite{S}.
\begin{enumerate}
\item
There is a $\C[[u]]$-linear,  $(u)$-adically continuous $L_\infty$ quasi-isomorphism of DG modules over the DGLA $\Gamma(X, \wedge^{\bullet+1}(T_X))$
$${\rm{CC}}^-_{-\bullet}(C^\infty (X))\isomoto  (\Omega^{-\bullet}[[u]], ud_{\rm{DR}})$$
whose reduction modulo $u$ is an $L_\infty$ quasi-isomorphism
$${C}_{-\bullet}(C^\infty (X), C^\infty (X))\isomoto  \Omega^{-\bullet}(X)$$
\item
There is a $\C[[u]]$-linear,  $(u)$-adically continuous $L_\infty$ quasi-isomorphism of DG modules over the DGLA $(\Gamma(X, \wedge^{\bullet+1}(T_X))[[\hbar]], [\pi,-]_{\rm{Sch}})$
$${{\rm{CC}}}^-_{-\bullet}(\bbA_\pi)\isomoto  (\Omega^{-\bullet}(X)[[\hbar, u]], L_\pi+ud_{\rm{DR}})$$
whose reduction modulo $u$ is an $L_\infty$ quasi-isomorphism
$${C}_{-\bullet}(\bbA^\pi,\bbA^\pi )\isomoto  (\Omega^{-\bullet}[[\hbar]], L_\pi)$$
\end{enumerate}
\end{thm}
\subsubsection{The complex analytic case}\label{fty for chains: complex case}  Let $\pi$ be a Maurer-Cartan element of the DGLA $(\hbar\Omega^{0,\bullet}(X, \wedge^{\bullet +1}T_X)[[\hbar]], {\overline{\partial}})$. Let $\bbA ^\pi _\Phi$ be the algebroid stack deformation corresponding to $
\pi$ by Theorem \ref{thm:cla def qua holo}. A Hochschild cochain complex $C^\bullet(\cA)$ of any algebroid stack $\cA$ was defined in \cite{BGNT}; the complexes $C_-{\bullet}(\cA), $ ${\operatorname{CC}}^-_{-\bullet}(\cA), $ and ${\operatorname{CC}}^{\rm{per}}_{-\bullet}(\cA)$ were defined in  \cite{chern}. As in the usual case, $C^{\bullet+1}(\cA)$ is a DGLA and the chain complexes are $DG$ modules over it.
\begin{thm}\label{thm:hocoho def qua analytic}
\begin{enumerate}
\item
There is $L_\infty$ quasi-isomorphism
$$\Omega^{0,\bullet}(X, \wedge ^{\bullet +1}(T_X))[[\hbar]], [\pi, -]_{\rm{Sch}}) \isomoto C^{\bullet+1}(\bbA_\Phi^\pi )$$
\item There is a $\C[[u]]$-linear $(u)$-adically continuous quasi-isomorphism of $L_\infty$ modules over the left hand side of the above formula
$${\operatorname{CC}}^-_{-\bullet}(\bbA_\Phi^\pi)\isomoto (\Omega^{0,\bullet}(X,\Omega_X^\bullet)[[\hbar, u]], {\overline{\partial}}+L_\pi+u\partial)$$
\end{enumerate}
\end{thm}
\subsection{Algebraic index theorem for deformations of symplectic structures}
Let $M$ be a smooth symplectic manifold. Let ${\mathbb{A}}^\hbar_M$
be a deformation quantization of a smooth symplectic manifold $M.$ Recall that
there exists canonical up to homotopy equivalence
quasi-isomorphism  
\begin{equation} \label{eq:tracedensity2}
\mu ^\hbar: \;CC_{\bullet}^-({\mathbb{A}}^\hbar(M))[\hbar^{-1}] \rightarrow
(\Omega^{2n-\bullet}(M)[[u]]((\hbar))] , i\hbar d)
\end{equation}
Localizing in $u$, we obtain a quasi-isomorphism
\begin{equation} \label{eq:tracedensity3}
\mu ^\hbar:
\;CC_{\bullet}^{\operatorname{per}}({\mathbb{A}}^\hbar(M))[\hbar^{-1}]
\rightarrow (\Omega^{2n-\bullet}(M)((u))((\hbar)), i\hbar d)
\end{equation}
(recall the notation from Definition \ref{dfn:((u))}).
\begin{definition} \label{dfn:tracedensity} The above morphisms are called
{\it the trace density
morphisms}.
\end{definition}
The index theorem compares the trace density morphism to the
principal symbol morphism. To define the latter, consider the
cyclic complex of the deformed algebra where the scalar ring is
$\Bbb C$ instead of ${\Bbb {C}}[[\hbar]]$. Consider the composition
$$CC_{\bullet}^-({\mathbb{A}}^\hbar M)) \to CC_{\bullet}^-(C^{\infty}(M)) \to $$
$$ \to (\Omega ^{\bullet}(M)[[u]], ud) \to (\Omega ^{\bullet}(M)[[u]][[\hbar]],
ud)$$
where the first morphism is reduction modulo $\hbar$, the second one
is $\mu$ from Theorem \ref{thm:HKRsmooth}, and the third one is
induced by the embedding ${\Bbb C} \to {\Bbb {C}}[[\hbar]]$. We will
denote this composition, followed by localization in $\hbar$, by
\begin{equation} \label{eq:principal symbol}
\mu: CC_{\bullet}^{\operatorname{per}}({\mathbb{A}}^\hbar(M)) \to
(\Omega ^{\bullet}(M)((u))((\hbar)), ud)
\end{equation}
To compare $\mu$ and $\mu ^\hbar$, let us identify the right hand
sides by the isomorphism
$$(\Omega ^{2n-\bullet}(M)((u))((\hbar)), i\hbar d) \to (\Omega
^{\bullet}(M)((u))((\hbar)), ud)$$
which is equal to $(\frac{t}{u})^{n-k}$ on $\Omega ^{k}(M)((u))((\hbar))$. After this identification, we obtain two
morphisms
$$\mu,\; \mu^t : CC_{\bullet}^{\operatorname{per}}({\mathbb{A}}^t(M)) \to
(\Omega ^{\bullet}(M)((u))((\hbar)), ud)$$
where the left hand side is defined as the periodic cyclic complex
with respect to the ground ring $\Bbb C$.
\begin{thm} \label{thm:index symplectic} At the level of cohomology,
$$\mu ^\hbar = \sum _{p=0}^{\infty} u^p (\widehat{A}(M)e^{\theta})_{2p} \cdot
\mu$$
where $\widehat{A}(M)$ is the $\widehat{A}$ class of the tangent
bundle of $M$ viewed as a complex bundle (with an almost complex
structure compatible with the symplectic form), and $\theta \in
{\frac{1}{i\hbar}}[\omega] + H^2 (M, {\Bbb C}[[\hbar]])$ is the
characteristic class of the deformation (cf. Theorem
\ref{thm:classification of symplectic deformations}).
\end{thm}
Note that the canonical trace $\operatorname {Tr} _{\operatorname
{can} }$ is the composition of $\mu ^\hbar$ with the integration
$\Omega ^{2n}((\hbar)) \to {\Bbb{C}}((\hbar))$. Let $P$ and
$Q$ be $N \times N$ matrices over ${\mathbb{A}}^t(M)$ such that
$P^2 = P$, $Q^2 = Q$, and $P-Q$ is compactly supported. Let $P_0$,
$Q_0$ be reductions of $P$, $Q$ modulo $\hbar$. They are idempotent
matrix-valued functions; their images $ P_0 {\Bbb C}^N$, $ Q_0
{\Bbb C}^N$ are vector bundles on $M$. Applying $\mu ^\hbar$ to the
the difference of Chern characters of $P$ and $Q$, we obtain the
following index theorem of Fedosov \cite{Fe} (cf. also
\cite{NT4}).
\begin{thm} \label{thm:fedosov index}
$${\operatorname {Tr}} _{\operatorname {can} }(P-Q)=\int _M
({{\operatorname{ch}}  (P_0 {\Bbb C}^N) -{\operatorname{ch}}  (Q_0 {\Bbb
C}^N)) \widehat{A}(M)e^{\theta}}
$$
\end{thm}
\subsection{Algebraic index theorem}\label{ss:Algebraic index theorem} The algebraic index theorem compares two morphisms from the periodic cyclic homology of a deformed algebra to the de Rham cohomology of the underlying manifold.
\subsubsection{The trace density map}\label{sss:TR density}
\begin{definition}\label{dfn:TR}
For a $C^\infty$ manifold $X,$ a formal Poisson structure $\pi$ on $X,$ and for the deformation quantization algebra $\bbA_\pi,$ define the trace density map
$${\operatorname{TR}}\colon {{\rm{CC}}}^{\rm{per}}_{-\bullet}(\bbA_\pi)\isomoto  (\Omega^{-\bullet}(X)[[\hbar]] ((u)), ud_{\rm{DR}})$$
to be the composition
$${{\rm{CC}}}^{\rm{per}}_{-\bullet}(\bbA^\pi)\isomoto  (\Omega^{-\bullet}(X)[[\hbar]] ((u)), L_\pi+ud_{\rm{DR}})\isomoto  (\Omega^{-\bullet}(X)[[\hbar]] ((u)), ud_{\rm{DR}})$$
where the map on the right is (the first component of) the first quasi-isomorphism (2), Theorem \ref{thm:hoho of def alg}, localized with respect to $u,$ and the map on the right is the isomorphism  $\exp(\frac{\iota_\pi}{u}).$
\end{definition}
\subsubsection{The principal symbol map}\label{sss:principal symbol} Denote by ${\operatorname{CC}}^{\operatorname{per}}_{-\bullet}(\bbA^\pi)_\C$ the periodic cyclic chain complex of $\bbA^\pi$ where the ring of scalars is defined as $\C,$ not $\C[[\hbar]].$ 
\begin{definition}\label{dfn:principal symbol}
Define the principal symbol map 
$$\sigma\colon {\operatorname{CC}}^{\operatorname{per}}_{-\bullet}(\bbA^\pi)_\C\isomoto (\Omega^{-\bullet}(X)((u)), ud_{\rm{DR}})$$
to be the composition  
$$ {\operatorname{CC}}^{\operatorname{per}}_{-\bullet}(\bbA^\pi)_\C\isomoto  {\operatorname{CC}}^{\operatorname{per}}_{-\bullet}(C^\infty (X))\isomoto (\Omega^{-\bullet}(X)((u)), ud_{\rm{DR}})$$
where the map on the left is induced by the corresponding morphism of algebras (reduction modulo $\hbar,$ a quasi-isomorphism by the Goodwillie rigidity theorem), and the map on the right is the HKR quasi-isomorphism.
\end{definition}
\begin{thm}\label{thm:alg index theorem}
For $a\in {\operatorname{HC}}^{\operatorname{per}}_{-\bullet}(\bbA^\pi)_\C,$
$${\operatorname{TR}}(a)=\iota(\sigma(a))\wedge {\sqrt{\widehat A}}(T_X)_u$$
where $\iota: \Omega^{-\bullet}(X)((u))\to \Omega^{-\bullet}(X)[[\hbar]]((u))$ is the inclusion and 
$${{\sqrt{\widehat A}}}(T_X)_u=({\sqrt{\widehat A}}(T_X))_{2p} u^{\pm p}$$
\end{thm}
\subsubsection{The complex analytic case} One defines, exactly as in \ref{sss:TR density} and in \ref{sss:principal symbol}, the quasi-isomorphisms
$${\operatorname{TR}}_\Phi\colon {\operatorname{CC}}^-_{-\bullet}(\bbA_\Phi^\pi)\isomoto (\Omega^{0,\bullet}(X,\Omega_X^\bullet)[[\hbar, u]], {\overline{\partial}}+u\partial)$$
and
$$\sigma_\Phi\colon {\operatorname{CC}}^-_{-\bullet}(\bbA_\Phi^\pi)_\C\isomoto (\Omega^{0,\bullet}(X,\Omega_X^\bullet)(( u)), {\overline{\partial}}+u\partial)$$
\begin{thm}\label{alg index thm anal case} For $a\in {\operatorname{HC}}^-_{-\bullet}(\bbA_\Phi^\pi)_\C,$
$${\operatorname{TR}}_\Phi(a)=i(\sigma_\Phi(a))\wedge ({\sqrt{\widehat A}}_\Phi(T_X))_u$$
\end{thm}

\subsubsection{Algebraic index theorem for traces} \label{sss:AIT traces}
\begin{thm}\label{thm:AIT traces} Let ${\mathbb A}^\pi$ be the deformation quantization of a $C^\infty$ manifold $M$ corresponding to a formal Poisson structure $\pi$. Let ${\rm{Tr}}\colon {\mathbb A}^\pi _c\to {\mathbb C}[[\hbar]]$ be a trace on the subalgebra of compactly supported functions. There exists a Poisson trace $\tau: C^\infty (M)[[\hbar]]\to {\mathbb C}[[\hbar]]$ with respect to $\pi$ such that, for any two idempotents $P$ and $Q$ in ${\rm{Matr}}_N ({\mathbb A}^\pi)$ such that $P-Q$ is compactly supported,
$${\rm{Tr}}(P-Q)=\langle \tau, {\rm{exp (\iota_{\pi})}}({\rm{ch}}(P_0-Q_0){\widehat{A}}^{\frac{1}{2}}(M))\rangle$$
where $P_0=P({\rm{mod}}\,\hbar), \; Q_0=Q({\rm{mod}}\,\hbar),$ and ${\rm{Tr}}$ is extended to the trace on the matrix algebra by ${\rm{Tr}}(a)=\sum {\rm{Tr}}(a_{ii}).$
\end{thm}

\end{document}